\documentclass[a4paper]{article}
\usepackage{latexsym,bm}
\usepackage{amssymb,amsmath,amsthm}
\usepackage[english,german]{babel}
\usepackage[colorlinks=true]{hyperref}
\usepackage{color}

\allowdisplaybreaks[4]

\newtheorem{theorem}{Theorem}[section]
\newtheorem{lemma}[theorem]{Lemma}
\newtheorem{remark}[theorem]{Remark}
\newtheorem{proposition}[theorem]{Proposition}
\newtheorem{definition}[theorem]{Definition}

\numberwithin{equation}{section}
\hypersetup{linkcolor=blue,urlcolor=red,citecolor=red}

\title{Feedback law
to stabilize linear infinite-dimensional  systems\thanks{This work was partially supported by the National Natural Science Foundation of China  under grant 11971022.}}

\author{Yaxing Ma\thanks{School of Mathematics, Tianjin University, Tianjin 300354, China (yaxingma@yeah.net).} \and Gengsheng Wang\thanks{Center for Applied Mathematics, Tianjin University, Tianjin, 300072, China (wanggs@yeah.net).} \and Huaiqiang Yu\thanks{School of Mathematics, Tianjin University, Tianjin 300354, China (huaiqiangyu@tju.edu.cn, huaiqiangyu@yeah.net).}}

\date{}
\begin{document}
\selectlanguage{english}
\maketitle
\begin{abstract}
    We design  a new  feedback law
    to stabilize  the linear infinite-dimensional control system, where the state operator
    generates  a $C_0$-group  and the control operator is unbounded.
     Our feedback law is    based on the integration of
     a mutated  Gramian operator-valued function. In the structure of the aforementioned mutated  Gramian operator, we utilize
       the weak observability inequality  in \cite{Trelat-Wang-Xu, Liu-Wang-Xu-Yu} and
                borrow some idea  used to construct  generalized Gramian operators
                in
                \cite{Komornik-1997, Urquiza, Vest}.
      Unlike most related works  where
    the exact controllability is required, we only assume  the above-mentioned weak observability inequality which is equivalent to
      the stabilizability of the system.

\end{abstract}
\vskip 8pt

 \noindent\textbf{Keywords.} Stabilizability,  feedback law, unbounded control operator, weak observability inequality
    \medskip

       \noindent \textbf{2010 AMS Subject Classifications.} 93B52, 93D05, 93D15
\vskip 10pt

%\tableofcontents
\section{Introduction}\label{sec_intro}
 %For the convenience of readers, we first introduce some notations.
\subsection{Notation}
 Let $\mathbb{R}^+:=[0,+\infty)$, $\mathbb{N}:=\{0,1,\ldots\}$ and $\mathbb{N}^+:=\{1,2,\ldots\}$. Given a Hilbert space $X$, we
 use $X'$, $\|\cdot\|_{X}$ and $\langle\cdot,\cdot\rangle_{X}$ to denote
 its dual space,  norm and inner product respectively; write  $\langle\cdot,\cdot\rangle_{X,X'}$
   for the dual product between $X$ and $X'$;  identify $X''$ (the dual space of $X'$) with $X$;   denote by $C(\mathbb{R}^+;X)$  the space of all continuous functions from $\mathbb{R}^+$ to $X$; write $I$ for the identity operator on $X$.
    When $L$ is  a densely defined and closed linear operator  on a Hilbert space $X$, we let $D(L):=\{x\in X:Lx\in X\}$ and $\|x\|_{D(L)}:=(\|x\|^2_{X}+\|Lx\|^2_{X})^{\frac{1}{2}}$ ($x\in D(L)$), which are the domain of $L$ and  the graph norm on $D(L)$ respectively;  use $L^*$
    to denote its adjoint operator, i.e., $\langle Lx,y\rangle_{X,X'}
    =\langle x,L^*y\rangle_{X,X'}$ $(x\in D(L), y\in D(L^*))$ (see \cite[Chapter 1, Section 1.10]{Pazy});
     use $\rho(L)$ to denote its resolvent set.
     When  $X_1$ and $X_2$ are two Hilbert spaces, we write
    $\mathcal{L}(X_1;X_2)$ for the space of all linear and bounded operators from $X_1$ to $X_2$, and further
    write  $\mathcal{L}(X_1):=\mathcal{L}(X_1;X_1)$.
    Given $F\in \mathcal{L}(X_1;X_2)$, we use $\|F\|_{\mathcal{L}(X_1;X_2)}$ and $F^*\in \mathcal{L}(X_2';X_1')$ to denote its operator norm and adjoint operator respectively.
    We use $C(\cdots)$ to denote a  constant that depends on what is enclosed in the brackets.
    
\subsection{Motivation }
    Stabilization is one of the most important objectives in control theory.
  There are two important subjects on stabilization for linear control systems: The first one is to find
sufficient conditions/equivalent conditions on  stabilizability, such as resolvent conditions   (see, for instance, \cite{Cearhart, Huang, Liu,  Pruss})
and weak observability inequalities (see \cite{ Liu-Wang-Xu-Yu, Trelat-Wang-Xu, Xu}).  The second one is to design feedback laws (see, for instance, \cite{Flandoli-Lasiecka-Triggiani, Kleinman, Komornik-1997, Lasiecka-Triggiani, Russell, Slemrod, Urquiza, Vest, Zabczyk}). For the latter,
we would like to mention two usual  methods: using  Riccati equations (see, for instance, \cite{Flandoli-Lasiecka-Triggiani, Lasiecka-Triggiani, Zabczyk}); using Gramian operators (see, for instance, \cite{Kleinman, Komornik-1997, Russell, Slemrod, Urquiza, Vest}). They all  have their own advantages and disadvantages. We aim to design feedback laws by the way using Gramian operators.

 It is well known that when matrices $A\in\mathbb{R}^{n\times n}$ and $B\in \mathbb{R}^{n\times m}$  ($n,m\in\mathbb{N}^+$) satisfy the Kalman controllability condition,
the matrix $K=-B^\top G_T^{-1}$ (where  $G_T:=\int_0^Te^{-At}BB^\top e^{-A^\top t}dt$, with $T>0$, is called a Gramian operator) is a feedback law stabilizing the system: $y'(t)=Ay(t)+Bu(t),\;t\geq 0$.
 (See, for instance,  \cite[Chapter 5, Section 5.7]{Sontag}.)
  Such idea has been extended to  the infinite-dimensional settings (see, for instance, \cite{Bourquin, Coron, Komornik-1997, Liu, Russell-1978, Slemrod, Urquiza, Vest}).
  To our best knowledge,   all existing  papers, which use Gramian operators to design  feedback laws,  need
 the following hypothesis (which is mainly used to ensure the invertibility of Gramian operators):
 \begin{itemize}
  \item [$\widehat{(H)}$] The  system  is exactly controllable at some time $T>0$.
\end{itemize}
    However, when researching how to design
         feedback laws, it is more natural
          to use stabilizability instead of controllability as an assumption.
           Indeed, there are many systems
          which are not  controllable but are stabilizable, even completely
     stabilizable   (see examples in \cite[Chapters  5]{Sontag} for finite-dimensional settings
     and in
     \cite{Huang-Wang-Wang, Liu-Wang-Xu-Yu, Trelat-Wang-Xu} for  infinite-dimensional settings).
On the other hand,    the stabilizability of a linear control   system
is equivalent to some weak observability inequality for its dual system  (see \cite{Liu-Wang-Xu-Yu, Trelat-Wang-Xu}).
{\it These motivate us to  design feedback laws via   modified Gramian operators, under assumption that the aforementioned
weak observability inequality holds. }

\subsection{System, hypotheses and definitions}

 {\bf System and hypotheses.}
 Let $H$ and $U$ be two Hilbert spaces. Consider the  control system:
\begin{equation}\label{yu-12-2-b-1}
    x'(t)=Ax(t)+Bu(t),\;\;t>0,
\end{equation}
where $u\in L^2(\mathbb{R}^+;U)$ and the pair $(A,B)$ verifies the following hypotheses:
\begin{enumerate}
     \item[$(H_1)$] The linear operator $A:D(A)\subset H\to H$ is the generator of a $C_0$-group $S(\cdot)$ on $H$;

    \item[$(H_2)$] The operator $B$ belongs to the space $\mathcal{L}(U;D(A^*)')$;

    \item[$(H_3)$] For any $T>0$, there is a constant $C(T)>0$ such that
\begin{equation}\label{yu-11-16-1}
    \int_0^T\|B^*S^*(t)\varphi\|_{U'}^2dt\leq C(T)\|\varphi\|_{H'}^2\;\;\mbox{for any}\;\;\varphi\in D(A^*).
\end{equation}
  %\item [$(H_4)$] There are constants $\alpha>0$, $D(\alpha)\geq 0$, $C(\alpha)\geq 1$ such that
%\begin{equation}\label{yu-11-16-6}
%    \|S^*(t)\varphi\|^2_{H'}\leq D(\alpha)\int_0^t\|B^*S^*(s)\varphi\|_{U'}^2ds
%    +C(\alpha)e^{-\alpha t}\|\varphi\|^2_{H'}\;\;\mbox{for any}\;\;t>0,\;\;
%    \varphi\in D(A^*).
%\end{equation}
\end{enumerate}

\begin{remark}\label{yu-remark-12-2-1}
   Several notes on the above  hypotheses and the system  are given.
\begin{enumerate}
\item [$(i)$] In general, the method by using Gramian operators to design feedback laws works only
for the linear control systems where the state operators generate $C_0$-groups. That is why we make the assumption $(H_1)$.
\item [$(ii)$]  The assumption $(H_2)$ has appeared in  many literature, where some specific examples
were given. (See, for instance, \cite{Komornik-1997, Lasiecka-Triggiani, Tucsnak-Weiss}).
From this assumption, we have $B^*\in \mathcal{L}(D(A^*);U')$. (Here, we identify $D(A^*)''$ with $D(A^*)$.)
The latter is equivalent to the existence of
    $\lambda\in \mathbb{C}$ and
    $E\in \mathcal{L}(U;H)$  such that $B^*=E^*(\lambda I+A)^*$ (see, for instance, \cite{ Tucsnak-Weiss, Vest}).
 Indeed,  we can choose $\lambda\in \rho(-A)$ and set $E:=(B^*((\lambda I+A)^{-1})^*)^*$.
 From now on, we fix $(\lambda, E)\in \mathbb{C}\times\mathcal{L}(U;H)$ satisfying
 $B^*=E^*(\lambda I+A)^*$.
\item [$(iii)$]  The assumption $(H_3)$ is called a regularity property in \cite{Coron,Lasiecka-Triggiani} or an admissibility condition
 in \cite{Tucsnak-Weiss}.
    With respect to this assumption, we have the following facts:
\begin{enumerate}
  \item [$(iii_{1})$] The condition $(H_3)$ is equivalent to the existence of  $T>0$ and  $C(T)>0$ so that (\ref{yu-11-16-1}) holds (see  \cite[Chapter 2, Section 2.3]{Coron});

  \item [$(iii_{2})$] If  $(H_1)$-$(H_3)$ are true,
   then
for any $T>0$, there is $C(T)>0$ so that
\begin{equation}\label{1.3,1-12}
    \int_{-T}^T\|B^*S^*(t)\varphi\|^2_{U'}dt \leq C(T)\|\varphi\|^2_{H'} \;\;\mbox{for any}
    \;\;\varphi\in D(A^*).
\end{equation}
    Consequently,  the  mapping $(\varphi\in D(A^*))\to (t\to B^*S^*(t)\varphi)\in L^2_{loc}(\mathbb{R};U')$ can be extended, in a unique way, to  a continuous operator, denoted by  $\widetilde{B^*S^*}(\cdot)$,  from $H'$ to $L^2_{loc}(\mathbb{R};U')$ (see
    \cite[Chapter 4, Section 4.3]{Tucsnak-Weiss}).
    Here, $L^2_{loc}(\mathbb{R};U')$ is regarded as a Fr\'{e}chet space with the seminorms:
    $\{\|\cdot\|_{L^2(-n,n;U')}\;:\;n\in\mathbb{N}^+\}$. 
\end{enumerate}
\item [$(iv)$]  The solutions of  (\ref{yu-12-2-b-1}) will be defined in the sense of transposition:
    a function $x(\cdot)\in C(\mathbb{R}^+;H)$ is called a solution of (\ref{yu-12-2-b-1}), with
    $u\in L^2(\mathbb{R}^+;U)$ and  $x(0)=x_0\in H$, if
\begin{equation*}\label{yu-nnn-2}
    \langle x(t),\varphi\rangle_{H,H'}
    =\langle x_0,S^*(t)\varphi\rangle_{H,H'}
    +\int_0^t\langle u(s),B^*S^*(t-s)\varphi\rangle_{U,U'}ds\;\;\mbox{for any}\;\;\varphi\in D(A^*), \;t\in \mathbb{R}^+.
\end{equation*}
    It deserves mentioning that under $(H_1)$-$(H_3)$, for any $x_0\in H$ and $u\in L^2(\mathbb{R}^+;U)$, the system (\ref{yu-12-2-b-1}), with initial condition $x(0)=x_0$, has a unique solution (see, for instance, \cite[Lemma 2.1]{Komornik-1997}).
    We denote it by $x(\cdot;x_0,u)$.

%\item [$(a_5)$]  In this paper, we do not identify $H$ (or $U$) with its dual space $H'$ (or $U'$).
%    One of the  reasons is that we want to put  boundary controlled hyperbolic-type equations into our framework.

%\item [$(a_3)$]  The inequality (\ref{yu-11-16-6}) in $(H_4)$
%   has appeared  in
%     \cite{Liu-Wang-Xu-Yu} where it is called  the weak observability inequality.
%         When $(H_1)$-$(H_3)$ hold,  we can show that the assumption
%     $(H_4)$ is equivalent to the following assumption:
%\begin{enumerate}
%  \item [$\widehat{(H_4)}$] There are $\delta\in(0,1)$, $T>0$ and $D(\delta,T)\geq 0$ such that
%  \begin{equation}\label{yu-11-16-5}
%    \|S^*(T)\varphi\|^2_{H'}\leq D(\delta,T)\int_0^T\|B^*S^*(s)\varphi\|_{U'}^2ds
%    +\delta\|\varphi\|^2_{H'}\;\;\mbox{for any}\;\;
%    \varphi\in D(A^*).
%\end{equation}
%\end{enumerate}
%    (See Proposition \ref{yu-proposition-12-9-1} in Section \ref{yu-subsec-12-9-1}).
%    The inequality (\ref{yu-11-16-5}) (which is also called a weak observability inequality) is built up in \cite{Trelat-Wang-Xu}.
%
\end{enumerate}
\end{remark}

\noindent {\bf Stabilizability and weak observability inequality.} We first recall that the operator $A$ has a unique extension $\widetilde{A}\in \mathcal{L}(H;D(A^*)')$ (see
    \cite[Proposition 2.10.4]{Tucsnak-Weiss} or \cite[Lemma 3.4]{Vest}), which is defined by
     \begin{equation}\label{1.3,202211}
     \langle \varphi, A^*\psi\rangle_{H,H'}=\langle \widetilde{A}\varphi, \psi\rangle_{D(A^*)',D(A^*)},\;\;
     \varphi\in H, \psi\in D(A^*).
     \end{equation}
         We now give the
     definition on the stabilizability for the system (\ref{yu-12-2-b-1}) which is quoted from \cite{Liu-Wang-Xu-Yu}.

\begin{definition}\label{yu-definition-11-17-1}
The  system (\ref{yu-12-2-b-1}) is said to be
  exponentially stabilizable (stabilizable, for simplicity), if there is a constant $\omega>0$,  a $C_0$-semigroup $\mathcal{S}(\cdot)$ on $H$ (with the generator $\mathcal{A}: D(\mathcal{A})\subset H\to H$) and an operator
  $K\in \mathcal{L}(D(\mathcal{A});U)$ so that
\begin{enumerate}
  \item [$(i)$] there is a constant $C_1\geq 1$ such that $\|\mathcal{S}(t)\|_{\mathcal{L}(H)}
  \leq C_1e^{-\omega t}$ for all  $t\in\mathbb{R}^+$;
  %\item [$(a_2)$] It holds that, for any $t\in\mathbb{R}^+$, $x_0\in D(\mathcal{A})$ and $\varphi\in D(A^*)$,
%\begin{equation}\label{yu-11-17-3}
%    \langle \mathcal{S}(t)x_0,\varphi\rangle_{H,H'}
%  =\langle x_0,S^*(t)\varphi_0\rangle_{H,H'}
%  +\int_0^t\langle K\mathcal{S}(s)x_0,B^*S^*(t-s)\varphi\rangle_{U,U'}ds;
%\end{equation}
\item[$(ii)$] for any $x\in D(\mathcal{A})$,  $\mathcal{A} x=(\widetilde{A}+BK)x$, with $\widetilde{A}$ given by \eqref{1.3,202211};
  \item [$(iii)$] there is a constant  $C_2\geq 0$ so that $\|K\mathcal{S}(\cdot)x\|_{L^2(\mathbb{R}^+;U)}
  \leq C_2\|x\|_H$ for any $x\in D(\mathcal{A})$.
\end{enumerate}

       The above $K$ and $\omega$
       are  called respectively a feedback law and  a stabilization decay rate (a decay rate, for simplicity).
   When the above $\omega$,   $\mathcal{S}(\cdot)$ and
  $K$ exist, we also say  that  $K$ is a feedback law stabilizing the system  (\ref{yu-12-2-b-1}) with the decay rate $\omega$.

\end{definition}

\begin{remark}\label{yu-remark-12-2-2}
    Several notes on Definition \ref{yu-definition-11-17-1} are as follows:
    \begin{enumerate}
\item[$(i)$]
The above definition comes originally from \cite{Flandoli-Lasiecka-Triggiani} which shows that
 the finite cost condition of the   LQ problem: $\inf_{u\in L^2(\mathbb{R}^+;U)}\int_0^{\infty}[\|x(t;x_0,u)\|_H^2+\|u(t)\|_U^2]dt$ implies the   stabilizability in the sense of Definition \ref{yu-definition-11-17-1}.
  It deserves mentioning that the above  stabilizability is equivalent to
  the above finite cost condition (see  \cite[Proposition 3.9]{Liu-Wang-Xu-Yu}).

\item [$(ii)$]  In general, $\widetilde{A}+BK$ is not the generator of the semigroup $\mathcal{S}(\cdot)$, except for the case $B\in\mathcal{L}(U;H)$. This operator  is only a densely defined restriction of such a generator
     (see \cite{Flandoli, Flandoli-Lasiecka-Triggiani, Komornik-1997, Lasiecka-Triggiani, Weiss-Rebarber}). The detailed explanation  is given  in the proof of our main theorem.

%\item [$(c_3)$] The quantity $\omega^*$ defined in (\ref{yu-12-19-1}) is the same as that defined by (4) in \cite{Trelat-Wang-Xu}.
%        It corresponds to   the definition of the growth bound for a $C_0$-semigroup (see, for instance,
%       \cite[Chapter 4, Section 2]{Engel-Nagel}). Notice that in general, $\omega^*$ is best stabilization decay rate of  the system (\ref{yu-12-2-b-1}) does not mean that $\omega^*$ is a decay rate of the system (\ref{yu-12-2-b-1}).
\item [$(iii)$] It was proved in \cite[Section 3.3]{Trelat-Wang-Xu} (see also \cite[Theorem 3.4]{Liu-Wang-Xu-Yu} for the complete stabilizability) that the stabilizability of the system (\ref{yu-12-2-b-1}) is equivalent to the following weak observability inequality for the dual system
    of  (\ref{yu-12-2-b-1}):

    There exists $\delta\in(0,1)$, $T>0$ and $C(\delta,T)\geq 0$ such that
  \begin{equation}\label{yu-11-16-5}
    \|S^*(T)\varphi\|^2_{H'}\leq C(\delta,T)\int_0^T\|B^*S^*(s)\varphi\|_{U'}^2ds
    +\delta\|\varphi\|^2_{H'}\;\;\mbox{for any}\;\;
    \varphi\in D(A^*).
\end{equation}
On the other hand,  \eqref{yu-11-16-5} is equivalent to what follows (see Proposition \ref{yu-proposition-12-9-1} in Section \ref{yu-subsec-12-9-1}):

There exists $\alpha>0$, $C_1(\alpha)\geq 0$ and $C_2(\alpha)\geq 1$ such that
\begin{equation}\label{yu-11-16-6}
    \|S^*(t)\varphi\|^2_{H'}\leq C_1(\alpha)\int_0^t\|B^*S^*(s)\varphi\|_{U'}^2ds
    +C_2(\alpha)e^{-\alpha t}\|\varphi\|^2_{H'}\;\;\mbox{for any}\;\;t>0,\;\;
    \varphi\in D(A^*),
\end{equation}
which  is also  called a weak observability inequality.
%Though \eqref{yu-11-16-5} is equivalent to \eqref{yu-11-16-6}, they have some difference:
%the former is static in time, while the latter is  dynamic in time. In certain cases, one may get more information from
% \eqref{yu-11-16-6}.

%\item[$(c_5)$] It is proved in \cite{Liu-Wang-Xu-Yu} that the complete stabilizability of system (\ref{yu-12-2-b-1}) is equivalent to that
   % for each $\alpha>0$, there exists $D(\alpha)\geq 0$ and $C(\alpha)\geq 1$ such that the inequality (\ref{yu-11-16-6}) holds.

  %\item [$(c_{33})$] in \cite{Trelat-Wang-Xu}, the authors have proved that
%$$
%    \omega^*=\sup\left\{-\frac{\ln \delta}{2T}:(\delta,T)\in(0,1)\times(0,+\infty):\exists\;D(\delta,T)\geq 0\;\mbox{s.t.}\;(\ref{yu-11-16-5})\;\mbox{holds}\right\}\footnote{It deseves
%    mentioning that, in \cite{Trelat-Wang-Xu}, $\overline{\omega}:=\inf\left\{\frac{\ln \delta}{2T}:(\delta,T)\in(0,1)\times(0,+\infty):\exists\;D(\delta,T)\geq 0\;\mbox{s.t.}\;(\ref{yu-11-16-5})\;\mbox{holds}\right\}$ is called the best stabilization decay rate. It is clear that
%    $\overline{\omega}=-\omega^*$.}.
%$$

\end{enumerate}
\end{remark}
Inspired by the note $(iii)$ in Remark \ref{yu-remark-12-2-2}, we further make the following hypothesis:
\begin{enumerate}
  \item [${(H_4)}$] There exists $\alpha>0$, $C_1(\alpha)\geq 0$ and $C_2(\alpha)\geq 1$ such that
  \eqref{yu-11-16-6} holds.
\end{enumerate}

\subsection{Main result}
    To state our main result, we need to introduce some operators. First of all, we let $J_1:U'\to U$
and $J_2:H'\to H$ be the    canonical isomorphisms given by
Riesz-Fr\'{e}chet representation theorem (see, for instance, \cite[Chapter 5, Section 5.2]{Brezis}). It should be noticed that $J_1$ and $J_2$ are conjugate-linear operators. Next, we
let $\alpha>0$, $C_1(\alpha)\geq 0$ and $C_2(\alpha)\geq 1$ be given in
    $(H_4)$. Now, for any $\varepsilon\in[0,\alpha)$,  $T>0$ and $t\in [0,T]$,  we define
    an  operator $\Lambda_{\alpha,\varepsilon,T}(t): H'\to H$ by
\begin{eqnarray}\label{yu-12-3-1}
    \langle \Lambda_{\alpha,\varepsilon,T}(t)\varphi,\psi\rangle_{H,H'}&=&
      C_1(\alpha)e^{\alpha T}\int_0^t e^{-(\alpha-\varepsilon)s}\langle J_1\widetilde{B^*S^*}(-s)\varphi,
      \widetilde{B^*S^*}(-s)\psi\rangle_{U,U'}ds\nonumber\\
      &\;&+C_2(\alpha)e^{-(\alpha-\varepsilon)t}\langle J_2S^*(-t)\varphi, S^*(-t)\psi\rangle_{H,H'}
      \;\;\mbox{for any}\;\;\varphi,\psi\in H',
\end{eqnarray}
where  $\widetilde{B^*S^*}(\cdot)$ is given in  $(iii_{2})$ of Remark \ref{yu-remark-12-2-1},
  and then define another operator $\Pi_{\alpha,\varepsilon,T}: H'\to H$ via
  \begin{equation}\label{yu-11-12-3}
    \langle \Pi_{\alpha,\varepsilon,T}\varphi,\psi\rangle_{H,H'}:=
    \int_0^T\left\langle\Lambda_{\alpha,\varepsilon,T}(t)\varphi ,\psi\right\rangle_{H,H'}dt
    \;\;\mbox{for any}\;\;\varphi,\psi\in H'.
\end{equation}
     It is clear that   both
    $\Lambda_{\alpha,\varepsilon,T}(t)$  and
     $\Pi_{\alpha,\varepsilon,T}$ are conjugate-linear. Moreover, we can show that
    $\Lambda_{\alpha,\varepsilon,T}(t)$  and
     $\Pi_{\alpha,\varepsilon,T}$ are bounded, and $\Pi_{\alpha,\varepsilon,T}^{-1}$ exists (see  Lemma \ref{yu-lemma-12-3-2}).
   Finally, for each $T$ satisfying
   \begin{equation}\label{1.7,12-27w}
   T\in \mathcal{I}_{\alpha}:=(\alpha^{-1}\ln[C_2(\alpha)], +\infty),
   \end{equation}
   we write $\widehat{\varepsilon}:=T^{-1}\ln[C_2(\alpha)]$, and then define
   an operator  $K_T: \Pi_{\alpha,\widehat{\varepsilon},T}[D(A^*)]\rightarrow U$ via
         \begin{equation}\label{1.9wang}
     K_T:=-TC_1(\alpha)e^{\alpha T}J_1B^*\Pi_{\alpha,\widehat{\varepsilon},T}^{-1}.
\end{equation}

\par
    The main result of this paper is  as follows:

\begin{theorem}\label{yu-theorem-12-3-1}
    Assume that $(H_1)$-$(H_4)$ are true.
    Then  for each $T$ satisfying (\ref{1.7,12-27w}),
    the operator $K_T$, defined by \eqref{1.9wang},  is a feedback law
    stabilizing  the system \eqref{yu-12-2-b-1}
    with the decay rate $\frac{1}{2}\left(\alpha-T^{-1}\ln[C_2(\alpha)]\right)$.
    \end{theorem}

\begin{remark}\label{remark1.5,1.7}
 Some notes on Theorem \ref{yu-theorem-12-3-1} are given.
\begin{enumerate}
  \item [$(i)$] Theorem \ref{yu-theorem-12-3-1}  gives a family of feedback laws $\{K_T\}_{T\in\mathcal{I}_\alpha}$ stabilizing the system \eqref{yu-12-2-b-1},
      and the decay rate corresponding to each $K_T$ has  an explicit expression. All coefficients in
the weak observability inequality \eqref{yu-11-16-6} appear in the expression of $K_T$.

 \item[$(ii)$] In  \eqref{1.9wang}, we only need
$\Pi_{\alpha,\varepsilon,T}^{-1}$ with ${\varepsilon}:=T^{-1}\ln[C_2(\alpha)]$, but in the proof of the main theorem, we will use the family $\{\Pi_{\alpha,\varepsilon,T}^{-1}\}_{{\varepsilon}\in [0,\alpha)}$.

  %\item[$(e_2)$] Theorem \ref{yu-theorem-12-3-1} means that, under the assumptions $(H_1)$-$(H_4)$, let $\alpha>0$ be given in $(H_4)$, then the optimal decay rate
%      of  the system (\ref{yu-12-2-b-1}) is greater than or equals to $\frac{\alpha}{2}$.
  \item[$(iii)$] We now explain our design of the feedback law $K_T$ as follows:
  First, based on the weak observability inequality in $(H_4)$,
 we define an operator $\Lambda_{\alpha,\varepsilon,T}(t)$ (given by \eqref{yu-12-3-1}),
 which can be treated as  a kind of
 mutated  Gramian operator.
  Thus, when $\varepsilon$ is fixed,
    $t\rightarrow \Lambda_{\alpha,\varepsilon,T}(t)$, $t\in [0,T]$, is a mutated  Gramian operator-valued function.
  Second, the operator $\Pi_{\alpha,\varepsilon,T}$ (given by \eqref{yu-11-12-3}) can be viewed as
  the integration of
  the aforementioned function. Thus, each $\Lambda_{\alpha,\varepsilon,T}(t)$ can be treated as a slice of $\Pi_{\alpha,\varepsilon,T}$.
  Third, the feedback $K_T$ (given by \eqref{1.9wang}) is built up
  with the aid of $\Pi_{\alpha,\varepsilon,T}$.

  It deserves mentioning the following two points: First, our structure is based on the
   assumption
  $(H_4)$ quantitatively. However, the feedback laws given in \cite{Komornik-1997, Russell, Slemrod, Urquiza, Vest}
  depend on the assumption of the exact controllability of \eqref{yu-12-2-b-1} qualitatively. (The latter will be explained in more detail in the next subsection.)
  Second, unlike works  \cite{ Komornik-1997, Russell, Slemrod, Urquiza, Vest}, we are not able to design a feedback law by only one slice $\Lambda_{\alpha,\varepsilon,T}(t)$. The main reason is that our assumption
  $(H_4)$ is weaker than the assumption of the observability used in \cite{Komornik-1997, Russell, Slemrod, Urquiza, Vest}.

  \item [$(iv)$] By the proof of Theorem  \ref{yu-theorem-12-3-1}, we would
  obtain a more general result (see Theorem \ref{yu-theorem-12-16-1}).

  \item [$(v)$]  The family $\{K_T\}_{T\in\mathcal{I}_\alpha}$ gives an approximate decay rate
  $\alpha/2$, where $\alpha$ has been fixed in  $(H_4)$. Thus, it seems  that our way to design
  feedback law can only give a fixed decay rate. Fortunately, this is not true.
    Indeed,
  we will show, in Section \ref{section411}, what follows: For each $\mu\in (0,\omega^*)$, where
  \begin{equation}\label{yu-12-19-1}
    \omega^*:=\sup\{\omega\in\mathbb{R}^+:\mbox{the system (\ref{yu-12-2-b-1})
    is stabilizable with decay rate}\;\omega\},
\end{equation}
 we can use our way to design a feedback law stabilizing
the system \eqref{yu-12-2-b-1} with the decay rate $\mu$ (see Theorem \ref{yu-theorem-12-20-1}).

\end{enumerate}

\end{remark}

\subsection{ Novelty and comparison with related works}\label{yu-sec-12-23-1}

For the studies relevant to our current work, we  recall the main results in \cite{Komornik-1997, Urquiza, Vest}.
The papers \cite{Komornik-1997, Vest}
 build up, for each $\omega>0$,
 a  generalized Gramian operator  $G_{T,\omega}: H\rightarrow H'$ (with $T>0$) via
\begin{equation}\label{1.11,1-7}
    \langle G_{T,\omega}\varphi,\psi\rangle_{H,H'}:=
    \int_0^{T_\omega}e_{\omega}(s)\langle J_1\widetilde{B^*S^*}(-s)\varphi, \widetilde{B^*S^*}(-s)\psi\rangle_{U,U'}ds
    \;\;\mbox{for any}\;\;\varphi, \psi\in H',
\end{equation}
   where   $T_{\omega}:=T+(2\omega)^{-1}$ and
$$
    e_{\omega}(s):=
    \begin{cases}
        e^{-2\omega s},&\mbox{if}\;\;s\in[0,T],\\
        2\omega e^{-2\omega T}(T_{\omega}-s),&\mbox{if}\;\;s\in [T,T_{\omega}],
    \end{cases}
$$
and   prove that  $K:=-J_1B^*G_{T, \omega}^{-1}$ is a feedback law stabilizing the system \eqref{yu-12-2-b-1}
with the decay rate $\omega$, where $G_{T,\omega}$ satisfies a Riccati equation.
      The paper \cite{Urquiza}   designs, for each $\omega>0$ large enough, a  generalized
     Gramian operator (which is originally from \cite{Russell} for some finite-dimensional systems):
\begin{equation}\label{1.12,1-7}
    \langle G_\omega\varphi,\psi\rangle_{H,H'}
    :=\int_0^\infty e^{-2\omega s}\langle J_1\widetilde{B^*S^*}(-s)\varphi,\widetilde{B^*S^*}(-s)\psi\rangle_{U,U'}ds
    \;\;\mbox{for any}\;\;\varphi,\psi\in H',
\end{equation}
    and   shows that $K:=-J_1B^*G^{-1}_{\omega}$ is a feedback law stabilizing  the system
    \eqref{yu-12-2-b-1}
    with the decay rate $(2\omega-g(-A))$ (where $g(-A):=\inf_{t>0}\frac{1}{t}\ln\|S(-t)\|_{\mathcal{L}(H)}$), where  $G_{\omega}$ satisfies a Lyapunov equation.

 In the above-mentioned papers \cite{Komornik-1997, Urquiza, Vest}, the assumption $\widehat{(H)}$
 (i.e., the system \eqref{yu-12-2-b-1} is exactly controllable at some time $T>0$) is
 necessary to  ensure the invertibility of the above  generalized
 Gramian operators, while  the
  corresponding  observability inequality is not fully utilized, more precisely, the coefficients in the observability inequality  does not appear  in the design of the feedback laws. Besides, either $G_{T,\omega}$ or $G_\omega$ corresponds to a slice $\Lambda_{\alpha,\varepsilon,T}(t)$.

     The novelties of this paper are as follows:
     \begin{itemize}
      \item  Our assumption $(H_4)$ is more natural and weaker than the above-mentioned $\widehat{(H)}$.
      Since $(H_4)$ cannot ensure the invertibility of the Gramian operators  $G_{T,\omega}$ and $G_\omega$ (given by
\eqref{1.11,1-7} and \eqref{1.12,1-7} respectively),  the method to design feedback laws in \cite{Komornik-1997, Urquiza, Vest} does not work for our case.

      \item  Our method to design feedback laws seems to be new from two perspective as follows:
            First, we replace
      the  generalized Gramian operator (in \cite{Komornik-1997, Urquiza, Vest}) with the integration of
      a mutated Gramian operator-valued
      function.
       It deserves mentioning that though  each slice $\Lambda_{\alpha,\varepsilon,T}(t)$  is invertible (see Lemma \ref{yu-12-3-10}),
       it does not work to replace
       $\Pi_{\alpha,\varepsilon,T}$
       by one slice $\Lambda_{\alpha,\varepsilon,T}(t)$ in \eqref{1.9wang}.
       Second, we use all information of the weak observability inequality. %(This has been explained in
%      the note $(iii)$ of Remark \ref{remark1.5,1.7}.)

    \item From perspective of stability, our design for feedback laws is reasonable in the sense:
        When
    $S(\cdot)$ is  stable, i.e., for some $\omega>0$ and $\widehat{C}(\omega)>0$, $\|S(t)\|_{\mathcal{L}(H)}\leq \widehat{C}(\omega)e^{-\omega t}$ for all $t\in\mathbb{R}^+$,
    the feedback law  should be $0$. This is consistent with our design. Indeed, in this case, we have
    $(H_4)$ with $\alpha=2\omega$, $C_1(\alpha)=0$ and $C_2(\alpha)=(\widehat{C}(\omega))^2$),
    which, along with  \eqref{1.9wang}, gives $K_T=0$.

        %{\color{blue}(Here, we note that, in this simple case, the decay rate given in Theorem \ref{yu-theorem-12-3-1} is $\omega-T^{-1}\ln C(\omega)$ not $\omega$ which is different if $C(\omega)\neq 1$. In our opinion, it is natural. Indeed, if we consider the control system in the finite-dimensional setting, i.e., $A,B$ are two appropriate matrices. Taking $\eta^*\in \sigma(A)$ so that $\mbox{Re}(\eta^*)\geq \mbox{Re}(\eta)$ for any $\eta\in \sigma(A)$. When $\mbox{Re}(\eta^*)<0$ and the geometric multiplicity of $\eta^*$ greater than $1$, we can conclude that, for each $\varepsilon\in(0,-\mbox{Re}(\eta^*))$, there is a $C(\varepsilon)\geq 1$ so that
%        $\|S(t)\|_{\mathcal{L}(H)}\leq C(\varepsilon)e^{-(-\mbox{Re}(\eta^*)+\varepsilon)t}$ for each $t\in\mathbb{R}^+$. In other words, when the geometric multiplicity of $\eta^*$ greater than $1$, the decay rate of $S(\cdot)$ has an error in itself.
%        The similar error will appear in any estimate of the decay rate when the feedback law is  a non-zero operator.)}

\end{itemize}

\subsection{Plan of this paper}

The rest of the paper is organized as follows:  Section \ref{yu-pre-12-23-1} shows some preliminaries;  Section \ref{yu-pro-12-23-1} proves the main result;  Section \ref{section411} presents  further studies;  Section \ref{sec-12-5} is appendix.

\section{Preliminaries}\label{yu-pre-12-23-1}
In this section, we suppose that $(H_1)$-$(H_4)$ hold and let $\alpha>0$, $C_1(\alpha)\geq 0$ and $C_2(\alpha)\geq 1$ be given in  $(H_4)$.
\begin{lemma}\label{yu-lemma-12-3-2}
    Given $T>0$ and $\varepsilon\in[0,\alpha)$. Let the operators
$\Lambda_{\alpha,\varepsilon,T}(t)$  (with $t\in [0,T]$) and $\Pi_{\alpha,\varepsilon,T}$, be defined by
(\ref{yu-12-3-1}) and (\ref{yu-11-12-3}) respectively. Then,  the following statements hold:
\begin{enumerate}
\item[(i)]
There is $C_0(T)>0$ so that   $\|\Lambda_{\alpha,\varepsilon,T}(t)\|_{\mathcal{L}(H';H)}\leq C_0(T)$ for all $t\in[0,T]$;

  \item [(ii)] The operator $\Pi_{\alpha,\varepsilon,T}$ is bounded;
  \item [(iii)] Both $\Lambda_{\alpha,\varepsilon,T}(t)$  and $\Pi_{\alpha,\varepsilon,T}$ are invertible.  Moreover,
\begin{equation}\label{yu-12-3-10}
    \langle \Lambda_{\alpha,\varepsilon,T}(t)\varphi,\varphi\rangle_{H,H'}
    \geq e^{\varepsilon t}\|\varphi\|_{H'}^2\;\;\mbox{and}\;\;\langle \Pi_{\alpha,\varepsilon,T}\varphi,\varphi\rangle_{H,H'}\geq T\|\varphi\|_{H'}^2\;\;\mbox{for all}\;\;
    \varphi\in H'.
\end{equation}

\end{enumerate}
\end{lemma}

\begin{proof}
Arbitrarily fix $T>0$ and $\varepsilon\in[0,\alpha)$. We begin with  proving $(i)$. Arbitrarily fix $\varphi, \psi\in D(A^*)$ and $t\in[0,T]$.
Since  $\widetilde{B^*S^*}(\cdot)=B^*S^*(\cdot)$ on $D(A^*)$ (see the note $(iii_{2})$ in Remark \ref{yu-remark-12-2-1}), it follows from (\ref{yu-12-3-1}) and \eqref{1.3,1-12}
 that
\begin{eqnarray*}\label{yu-12-3-12}
    &\;&|\langle \Lambda_{\alpha,\varepsilon,T}(t)\varphi, \psi\rangle_{H,H'}|\nonumber\\
    &\leq& C_1(\alpha)e^{\alpha T}\Big(\int_0^T\|B^*S^*(-s)\varphi\|_{U'}^2ds\Big)^{\frac{1}{2}}
    \Big(\int_0^T\|B^*S^*(-s)\psi\|^2_{U'}ds\Big)^{\frac{1}{2}}
       +C_2(\alpha)\|S^*(-t)\varphi\|_{H'}\|S^*(-t)\psi\|_{H'}\nonumber\\
    &\leq& \Big(C_1(\alpha)e^{\alpha T}C(T)+C_2(\alpha)\Big(\sup_{s\in [0,T]}\|S^*(-s)\|_{\mathcal{L}(H';H')}\Big)^2\Big)
    \|\varphi\|_{H'}\|\psi\|_{H'}.
\end{eqnarray*}
     This, along with  the density
    of $D(A^*)$ in $H'$, leads to  $(i)$
    with
    $$
    C_0(T):=C_1(\alpha)e^{\alpha T}C(T)+C_2(\alpha)\big(\sup_{s\in [0,T]}\|S^*(-s)\|_{\mathcal{L}(H';H')}\big)^2.
    $$
\par
    To  show $(ii)$, we arbitrarily fix  $\varphi,\psi\in H'$.
     It follows from (\ref{yu-12-3-1}) that the function $t\rightarrow \langle\Lambda_{\alpha, \varepsilon, T}(t)\varphi,\psi\rangle_{H,H'}$, $t\in [0,T]$,  is continuous, so is integrable. This, along with  \eqref{yu-11-12-3} and the property $(i)$ in this lemma, yields
$$
|\langle\Pi_{\alpha,\varepsilon,T}\varphi,\psi\rangle_{H,H'}|\leq TC_0(T)\|\varphi\|_{H'}\|\psi\|_{H'},
$$
    which leads to $(ii)$.
\par
    We now prove $(iii)$. Because $S(\cdot)$ is a group (see  $(H_1)$),
     the inequality (\ref{yu-11-16-6}) (which is true by $(H_4)$) is equivalent to
\begin{equation*}\label{yu-11-4-6-11}
    \|\varphi\|^2_{H'}\leq C_1(\alpha)\int_0^t\|B^*S^*(-s)\varphi\|^2_{U'}ds+C_2(\alpha)e^{-\alpha t}\|S^*(-t)\varphi\|_{H'}^2\;\;\mbox{for any}\;\;t\in\mathbb{R}^+, \; \varphi\in D(A^*).
\end{equation*}
      It follows  that when
    $\varepsilon\in[0,\alpha)$ and $t\in[0,T]$,
\begin{eqnarray*}\label{yu-11-15-3}
    e^{\varepsilon t}\|\varphi\|^2_{H'}\leq C_1(\alpha)e^{\alpha T}
    \int_0^{t}e^{-(\alpha-\varepsilon)s}\|B^*S^*(-s)\varphi\|_{U'}^2ds
    +C_2(\alpha)e^{-(\alpha-\varepsilon)t}\|S^*(-t)\varphi\|_{H'}^2\;\;
    \mbox{for any}\;\;\varphi\in D(A^*),
\end{eqnarray*}
    which, together with (\ref{yu-12-3-1}), yields  the first estimate in (\ref{yu-12-3-10}) with $\varphi\in D(A^*)$. This, along with the density of $D(A^*)$ in $H'$, shows the first estimate in (\ref{yu-12-3-10}) with $\varphi\in H'$.
        Next, the second estimate in  (\ref{yu-12-3-10}) follows from the first one and (\ref{yu-11-12-3}).
        Finally, it follows from (\ref{yu-12-3-10}), the claims $(i)$ and $(ii)$, and the Lax-Milgram theorem that both $\Lambda_{\alpha,\varepsilon,T}(t)$ and $\Pi_{\alpha,\varepsilon,T}$ are invertible.

   Hence, we complete the proof of Lemma \ref{yu-lemma-12-3-2}.
\end{proof}

\begin{proposition}\label{yu-proposition-11-12-1}
Let $T>0$ and $\varepsilon\in[0,\alpha)$. Then the following conclusions are true:
\begin{enumerate}
\item[(i)] Let the operator $\Pi_{\alpha,\varepsilon,T}$ be given by (\ref{yu-11-12-3}). If let $\mathcal{X}:=\Pi_{\alpha,\varepsilon,T}$, then $\mathcal{X}$ is a solution of the following Lyapunov equation:
\begin{eqnarray}\label{yu-11-12-4}
    &\;&\langle \mathcal{X}A^*\varphi,\psi\rangle_{H,H'}
    +\langle \mathcal{X}\varphi,A^*\psi\rangle_{H,H'}
    -TC_1(\alpha)e^{\alpha T}\langle J_1B^*\varphi,B^*\psi\rangle_{U,U'}\nonumber\\
    &=&-(\alpha-\varepsilon)\langle\mathcal{X}\varphi,\psi\rangle_{H,H'}-\langle Q_{\alpha,\varepsilon,T}\varphi,\psi\rangle_{H,H'}\;\;\mbox{for any}\;\;\varphi,\psi\in D(A^*),
\end{eqnarray}
    where the bounded operator $Q_{\alpha,\varepsilon,T}$ is defined by
\begin{eqnarray}\label{yu-11-13-1-b}
    Q_{\alpha,\varepsilon,T}
    :=\Lambda_{\alpha,\varepsilon,T}(T)-C_2(\alpha) J_2,\;\;\mbox{with}\;\;\Lambda_{\alpha,\varepsilon,T}(T)
    \;\mbox{given by}\; (\ref{yu-12-3-1}).
\end{eqnarray}
   \item[(ii)] When  $(\varepsilon, T)$ verifies
     \begin{equation}\label{yu-mmmm-1}
\begin{cases}
    \varepsilon\in(0,\alpha)\;\;\mbox{and}\;\;T\geq \varepsilon^{-1}\ln [C_2(\alpha)], &\mbox{if}\;\;
    C_2(\alpha)>1,\\
    \varepsilon\in[0,\alpha)\;\;\mbox{and}\;\;T>0, &\mbox{if}\;\;C_2(\alpha)=1,
\end{cases}
\end{equation}
     the operator  $Q_{\alpha,\varepsilon,T}$, given by \eqref{yu-11-13-1-b},
           is non-negative in the sense of $\langle Q_{\alpha,\varepsilon,T}\varphi,\varphi\rangle_{H,H'}\geq 0$ for any $\varphi\in H'$.
     \end{enumerate}
\end{proposition}
\begin{proof}
We begin with showing $(i)$.
     Let
     \begin{equation}\label{2.5,1-9}
     \widehat{A}_{\alpha,\varepsilon}:=A^*+\frac{1}{2}(\alpha-\varepsilon)I,\;\;\mbox{with}
     \;\;D(\widehat{A}_{\alpha,\varepsilon})=D(A^*).
     \end{equation}
          Write
        $\widehat{S}_{\alpha,\varepsilon}(\cdot)$ for the $C_0$-group
    generated by $\widehat{A}_{\alpha,\varepsilon}$.
    Two observations are given in order: First,
    since
$\widetilde{B^*S^*}(\cdot)=B^*S^*(\cdot)$ on $D(A^*)$ (see the note $(iii_{2})$ in Remark \ref{yu-remark-12-2-1}),
    it follows by  (\ref{yu-12-3-1})  that
\begin{eqnarray}\label{yu-12-3-12}
    \langle\Lambda_{\alpha,\varepsilon,T}(t)\varphi,\psi\rangle_{H,H'}
    &=&C_1(\alpha)e^{\alpha T}\int_0^t\langle J_1B^*\widehat{S}_{\alpha,\varepsilon}(-s)\varphi,
    B^*\widehat{S}_{\alpha,\varepsilon}(-s)\psi\rangle_{U,U'}\nonumber\\
    &\;&+C_2(\alpha)\langle J_2\widehat{S}_{\alpha,\varepsilon}(-t)\varphi,
    \widehat{S}_{\alpha,\varepsilon}(-t)\psi\rangle_{H,H'},\;
    \varphi,\psi\in D(A^*).
\end{eqnarray}
   Second, by the note $(ii)$ of Remark \ref{yu-remark-12-2-1} and the first observation above, we see that
   when  $\varphi,\psi\in D((A^*)^2)$ and $s\in[0,T]$,
$$
    \langle J_1B^*\widehat{S}_{\alpha,\varepsilon}(-s)\varphi,B^*
    \widehat{S}_{\alpha,\varepsilon}(-s)\psi\rangle_{H,H'}
    =\langle J_1E^*\widehat{S}_{\alpha,\varepsilon}(-s)(\lambda I+A)^*\varphi, E^*
    \widehat{S}_{\alpha,\varepsilon}(-s)(\lambda I+A)^*\psi\rangle_{H,H'},
$$
from which, it follows that for any $\varphi,\psi\in D((A^*)^2)$, the function $s\rightarrow \langle J_1B^*\widehat{S}_{\alpha,\varepsilon}(-s)\varphi,B^*
\widehat{S}_{\alpha,\varepsilon}(-s)\psi\rangle_{H,H'}$  is continuously differentiable over $[0,T]$.
\par
    We now arbitrarily fix $\varphi,\psi\in D((A^*)^2)$. On the one hand, by the second observation above, we find
    \begin{eqnarray*}\label{yu-11-4-6-1}
    &\;&C_1(\alpha)e^{\alpha T}\int_0^t\frac{d}{ds}\left(\langle J_1B^*
    \widehat{S}_{\alpha,\varepsilon}(-s)\varphi, B^*
    \widehat{S}_{\alpha,\varepsilon}(-s)\psi\rangle_{U,U'}\right)ds\nonumber\\
    &=&C_1(\alpha)e^{\alpha T}\left(\langle J_1B^*\widehat{S}_{\alpha,\varepsilon}(-t)\varphi, B^*\widehat{S}_{\alpha,\varepsilon}(-t)\psi\rangle_{U,U'}-\langle J_1B^*\varphi,B^*\psi\rangle_{U,U'}\right),\; t\in[0,T].
\end{eqnarray*}
  On the other hand, it follows  from (\ref{yu-12-3-12}) that for each $t\in[0,T]$,
\begin{eqnarray*}\label{yu-11-12-5}
    &\;&C_1(\alpha)e^{\alpha T}\int_0^t\frac{d}{ds}\left(\langle J_1 B^*
    \widehat{S}_{\alpha,\varepsilon}(-s)\varphi, B^*
    \widehat{S}_{\alpha,\varepsilon}(-s)\psi\rangle_{U,U'}\right)ds\nonumber\\
    &=&-C_1(\alpha)e^{\alpha T}\biggl(\int_0^t\langle J_1B^*
    \widehat{S}_{\alpha,\varepsilon}(-s)\widehat{A}_{\alpha,\varepsilon}\varphi, B^*
    \widehat{S}_{\alpha,\varepsilon}(-s)\psi\rangle_{U,U'}ds\nonumber\\
    &\;&
   +\int_0^t\langle J_1 B^*\widehat{S}_{\alpha,\varepsilon}(-s)\varphi, B^*
   \widehat{S}_{\alpha,\varepsilon}(-s)
   \widehat{A}_{\alpha,\varepsilon}\psi\rangle_{U,U'}ds\biggl)\nonumber\\
    &=&-\langle \Lambda_{\alpha,\varepsilon,T}(t)
    \widehat{A}_{\alpha,\varepsilon}\varphi,\psi\rangle_{H,H'}
    -\langle \Lambda_{\alpha,\varepsilon,T}(t)\varphi,
    \widehat{A}_{\alpha,\varepsilon}\psi\rangle_{H,H'}\nonumber\\
    &\;&+C_2(\alpha)\left(\langle J_2\widehat{A}_{\alpha,\varepsilon}
    \widehat{S}_{\alpha,\varepsilon}(-t)\varphi,
    \widehat{S}_{\alpha,\varepsilon}(-t)\psi\rangle_{H,H'}+\langle J_2
    \widehat{S}_{\alpha,\varepsilon}(-t)\varphi,
    \widehat{A}_{\alpha,\varepsilon}
    \widehat{S}_{\alpha,\varepsilon}(-t)\psi\rangle_{H,H'}\right)\nonumber\\
    &=&-\langle \Lambda_{\alpha,\varepsilon,T}(t)
    \widehat{A}_{\alpha,\varepsilon}\varphi,\psi\rangle_{H,H'}-\langle \Lambda_{\alpha,\varepsilon,T}(t)\varphi,\widehat{A}_{\alpha,\varepsilon}\psi\rangle_{H,H'}
    \nonumber\\
    &\;&
    -C_2(\alpha)\left[\frac{d}{ds}
    \left(\langle J_2\widehat{S}_{\alpha,\varepsilon}(-s)\varphi,
    \widehat{S}_{\alpha,\varepsilon}(-s)\psi
    \rangle_{H,H'}\right)\right]_{s=t}.
\end{eqnarray*}
     The above two equalities imply that for each $t\in[0,T]$,
\begin{eqnarray*}\label{yu-11-12-6}
    &\;&C_1(\alpha)e^{\alpha T}\left(\langle J_1 B^*\widehat{S}_{\alpha,\varepsilon}(-t)\varphi, B^*\widehat{S}_{\alpha,\varepsilon}(-t)\psi\rangle_{U,U'}-\langle J_1 B^*\varphi,B^*\psi\rangle_{U,U'}\right)\nonumber\\
    &=&-\langle \Lambda_{\alpha,\varepsilon,T}(t)\widehat{A}_{\alpha,\varepsilon}
    \varphi,\psi\rangle_{H,H'}-\langle \Lambda_{\alpha,\varepsilon,T}(t)\varphi,
    \widehat{A}_{\alpha,\varepsilon}\psi\rangle_{H,H'}\nonumber\\
    &\;&-C_2(\alpha)\left[\frac{d}{ds}\left(\langle J_2
    \widehat{S}_{\alpha,\varepsilon}(-s)\varphi,\widehat{S}_{\alpha,\varepsilon}(-s)\psi
    \rangle_{H,H'}\right)\right]_{s=t}.
\end{eqnarray*}
    Integrating the above equality with respect to $t$ over $[0,T]$ and using (\ref{yu-11-12-3}), we obtain
\begin{eqnarray}\label{yu-11-12-7}
    &\;&\langle \Pi_{\alpha,\varepsilon,T}
    \widehat{A}_{\alpha,\varepsilon}\varphi,\psi\rangle_{H,H'}+\langle \Pi_{\alpha,\varepsilon,T}\varphi,\widehat{A}_{\alpha,\varepsilon}\psi\rangle_{H,H'}-T
    C_1(\alpha)e^{\alpha T}\langle J_1B^*\varphi,B^*\psi\rangle_{U,U'}\nonumber\\
    &=& -C_1(\alpha)e^{\alpha T}\int_0^T\langle J_1 B^*
    \widehat{S}_{\alpha,\varepsilon}(-t)\varphi,
    B^*\widehat{S}_{\alpha,\varepsilon}(-t)\psi\rangle_{U,U'}dt\nonumber\\
    &\;&-C_2(\alpha)\int_0^T\frac{d}{dt}\left(\langle J_2\widehat{S}_{\alpha,\varepsilon}(-t)\varphi,
    \widehat{S}_{\alpha,\varepsilon}(-t)\psi\rangle_{H,H'}\right)dt.
\end{eqnarray}
        Meanwhile, it follows from (\ref{yu-12-3-12})  that
\begin{eqnarray*}\label{yu-11-12-8}
    &\;&C_2(\alpha)\int_0^T\frac{d}{dt}\left(\langle J_2
    \widehat{S}_{\alpha,\varepsilon}(-t)\varphi,
    \widehat{S}_{\alpha,\varepsilon}(-t)\psi\rangle_{H,H'}\right)dt
    \nonumber\\
    &=&C_2(\alpha)\langle J_2\widehat{S}_{\alpha,\varepsilon}(-T)\varphi,
    \widehat{S}_{\alpha,\varepsilon}(-T)\psi\rangle_{H,H'}-C_2(\alpha)\langle J_2\varphi,\psi\rangle_{H,H'}\nonumber\\
    &=& \langle \Lambda_{\alpha,\varepsilon,T}(T)\varphi,\psi\rangle_{H,H'} -C_1(\alpha)e^{\alpha T}\int_0^T\langle J_1B^*\widehat{S}_{\alpha,\varepsilon}(-t)\varphi, B^*\widehat{S}_{\alpha,\varepsilon}(-t)\psi\rangle_{U,U'}dt-C_2(\alpha)\langle J_2\varphi,\psi\rangle_{H,H'}.
\end{eqnarray*}
    Replacing the above equality to (\ref{yu-11-12-7}), we get
\begin{eqnarray*}\label{yu-11-12-9}
   &\;& \langle \Pi_{\alpha,\varepsilon,T}
   \widehat{A}_{\alpha,\varepsilon}\varphi,\psi\rangle_{H,H'}+\langle \Pi_{\alpha,\varepsilon,T}\varphi,\widehat{A}_{\alpha,\varepsilon}\psi\rangle_{H,H'}-T
    C_1(\alpha) e^{\alpha T}\langle J_1B^*\varphi,B^*\psi\rangle_{U,U'}\nonumber\\
   &=& C_2(\alpha)\langle J_2\varphi,\psi\rangle_{H,H'}
    -\langle \Lambda_{\alpha,\varepsilon,T}(T)\varphi, \psi\rangle_{H,H'},
\end{eqnarray*}
    which, together with \eqref{2.5,1-9}
     and (\ref{yu-11-13-1-b}), shows that  $\Pi_{\alpha,\varepsilon,T}$ verifies the equation (\ref{yu-11-12-4}), with $\varphi,\psi\in D((A^*)^2)$.
      This, along with  the density of $D((A^*)^2)$ in $D(A^*)$, shows
      that
    $\mathcal{X}:=\Pi_{\alpha,\varepsilon,T}$ is a solution of the equation (\ref{yu-11-12-4}).
\par
  We next show $(ii)$. Indeed, in the case that $C_2(\alpha)>1$, we see from (\ref{yu-12-3-10}) that  for any
    $\varepsilon\in(0,\alpha)$ and  $T\geq \varepsilon^{-1}\ln [C_2(\alpha)]$,
\begin{equation*}\label{yu-11-15-100}
    \langle Q_{\alpha,\varepsilon,T}\xi,\xi\rangle_{H,H'}=\langle \Lambda_{\alpha,\varepsilon,T}(T)\xi, \xi\rangle_{H,H'}-C_2(\alpha)\|\xi\|^2_{H'}\geq (e^{\varepsilon T}-C_2(\alpha))\|\xi\|_{H'}^2\geq 0,\;\;\xi\in H',
\end{equation*}
  while in the case that $C_2(\alpha)=1$, we use (\ref{yu-12-3-10}) to get that for any $\varepsilon\in[0,\alpha)$ and
    $T>0$,
\begin{equation*}\label{yu-11-15-100-bb}
    \langle Q_{\alpha,\varepsilon,T}\xi,\xi\rangle_{H,H'}=\langle \Lambda_{\alpha,\varepsilon,T}(T)\xi, \xi\rangle_{H,H'}-\|\xi\|^2_{H'}\geq  0,\;\;\xi\in H'.
\end{equation*}
   These imply that  $Q_{\alpha,\varepsilon,T}$,
   with  $(\varepsilon,T)$ satisfying \eqref{yu-mmmm-1},
    is non-negative.

    Thus we finish the proof of  Proposition \ref{yu-proposition-11-12-1}.
\end{proof}

\begin{remark}\label{reamrk2.3w}
First, in the proof of Proposition \ref{yu-proposition-11-12-1}, we used the weak observability inequality in $(H_4)$. Second, in the proof of Theorem \ref{yu-theorem-12-3-1}, Proposition \ref{yu-proposition-11-12-1} plays an important  role. Third, in the proof of Theorem \ref{yu-theorem-12-3-1},  we also borrowed  another idea,
  which was widely used in the related works (see, for instance,
\cite{Flandoli-Lasiecka-Triggiani, Flandoli, Vest}) and can be  explained  as follows:  By Proposition \ref{yu-proposition-11-12-1}, $\Pi_{\alpha,\varepsilon,T}$ satisfies
    the Laypunov equation (\ref{yu-11-12-4}), which  can be written formally as
\begin{equation*}\label{yu-11-22-1}
    \Pi_{\alpha,\varepsilon,T}^{-1}\left(A-TC_1(\alpha)e^{\alpha T}BJ_1B^*\Pi_{\alpha,T}^{-1}\right)
    \Pi_{\alpha,\varepsilon,T}=-A^*-P_{\alpha,\varepsilon,T},
\end{equation*}
where
\begin{equation}\label{yu0-bb-1}
    P_{\alpha,\varepsilon,T}:=(\alpha-\varepsilon)I+\Pi_{\alpha,\varepsilon,T}^{-1}
    Q_{\alpha,\varepsilon,T}.
\end{equation}
(The existence of $\Pi^{-1}_{\alpha,\varepsilon,T}$ is ensured by $(iii)$ of Lemma \ref{yu-lemma-12-3-2}.) In this sense, the operators $-A^*-P_{\alpha,\varepsilon,T}$ and $A-TC_1(\alpha,T)BJ_1B^*\Pi_{\alpha,T}^{-1}$
     are ``conjugated'' each other. Thus one can obtain a $C_0$-group on $H$ generated by $A-TC_1(\alpha)e^{\alpha T}BJ_1B^*\Pi_{\alpha,T}^{-1}$ formally, through using the $C_0$-group on $H'$ generated by $-A^*-P_{\alpha,\varepsilon,T}$.

    Now back to our case. Write  $\mathcal{V}_{\alpha,\varepsilon,T}(\cdot)$ for the
    $C_0$-group on $H'$, generated by $\Delta_{\alpha,\varepsilon,T}:=-A^*-P_{\alpha,\varepsilon,T}$, with its domain $D(\Delta_{\alpha,\varepsilon,T})$ which is the same as $D(A^*)$. Here, we notice that
    $P_{\alpha,\varepsilon,T}\in\mathcal{L}(H')$.
     Then by the constant variation formula, we have
     \begin{equation}\label{yu-11-23-1}
    \mathcal{V}_{\alpha,\varepsilon,T}(t)\varphi
    =S^*(-t)\varphi-\int_0^tS^*(s-t)P_{\alpha,\varepsilon,T}
    \mathcal{V}_{\alpha,\varepsilon,T}(s)\varphi ds\;\;\mbox{for any}\;\;t\in\mathbb{R}, \; \varphi\in H'.
\end{equation}

\end{remark}

%\par
%    In the rest of this subsection, we shall present the proof of Theorem
%    \ref{yu-11-theorem-29-1}. The main idea for its proof comes originally from
%     \cite{Flandoli} for differential Riccati equation. {\color{blue}In \cite{Vest}, the author gives a
%     proof for algebraic Riccati equation. Unfortunately,  there are many misprints in the proof of \cite{Vest}. Moreover, in \cite{Vest}, the author only checked the exponential decay of the corresponding closed system, has not check all conditions in $(i)$ of Definition \ref{yu-definition-11-17-1}.} Therefore, we give the complete proof in our paper.
%     Indeed, the different between our proof and \cite{Vest} can be stated as the following four aspects:
%\begin{itemize}
%  \item We added a new lemma  to replace Proposition 2.3 of \cite{Vest} (see Lemma \ref{yu-lemma-11-23-2} blew);
%  \item We rewritten the proof of Lemma 3.2 of \cite{Vest} based on our new lemma (i.e., Lemma \ref{yu-lemma-11-23-2} below);
%  \item
%      We revised
%     the proofs of Proposition 4.1 and 4.2 of \cite{Vest} (see Steps 3, 4 in the proof of Theorem \ref{yu-11-theorem-29-1});
%  \item We checked the exponential stabilizability in the the sense of $(i)$ in Definition \ref{yu-definition-11-17-1} (see Step 5 in the proof of Theorem \ref{yu-11-theorem-29-1}).
%\end{itemize}
\par
    The next  two lemmas will be used  in the proof Theorem \ref{yu-theorem-12-3-1}.
    For the first one, we  did not find any exact version  in  published papers, while for the second one,
           a  similar result  was given in \cite[Lemma 3.3]{Vest}, however,  in its proof, there are some places that we do not understand. So we give their proofs in Section \ref{sec-12-5}.

\begin{lemma}\label{yu-lemma-11-23-2}
Given  $\gamma>0$, $M\in \mathcal{L}(H')$ and $\varphi\in D(A^*)$, let $w(t;\varphi):=\int_0^tS^*(s-t)M\mathcal{V}_{\alpha,\varepsilon,T}(s)\varphi ds$, $t\in[-\gamma,\gamma]$. Then
the following conclusions are true:
 \begin{enumerate}
\item[(i)] For any $t\in[-\gamma,\gamma]$,  $w(t;\varphi)\in D(A^*)$;
\item[(ii)] There is a constant $C(\gamma)>0$ (independent of $M$ and $\varphi$ but depending on $\gamma$) so that
\begin{equation}\label{yu-11-25-bbb-1}
    \int_{-\gamma}^\gamma\|B^*w(t;\varphi)\|^2_{U'}dt\leq C(\gamma)
    \Big(\|M\varphi\|_{H'}^2+
    \int_{-\gamma}^{\gamma}
    \left(\|M\mathcal{V}_{\alpha,\varepsilon,T}(t)\varphi\|_{H'}^2
    +\|M\mathcal{V}_{\alpha,\varepsilon,T}(t)\Delta_{\alpha,\varepsilon,T}\varphi\|_{H'}^2
    \right)dt\Big).
\end{equation}
 \end{enumerate}
\end{lemma}

\begin{lemma}\label{yu-lemma-11-23-1}
   For any $\varphi, \psi\in D(A^*)$ and $t\in\mathbb{R}$,
\begin{eqnarray}\label{yu-11-23-2}
    \langle \Pi_{\alpha,\varepsilon,T}\varphi,\psi\rangle_{H,H'}
    &=&\langle \Pi_{\alpha,\varepsilon,T}\mathcal{V}_{\alpha,\varepsilon,T}(t)\varphi,S^*(-t)\psi
    \rangle_{H,H'}\nonumber\\
    &\;&+TC_1(\alpha)e^{\alpha T}\int_0^t\langle J_1B^*\mathcal{V}_{\alpha,\varepsilon,T}(s)\varphi,B^*S^*(-s)\psi\rangle_{U,U'}ds.
\end{eqnarray}
    \end{lemma}

\section{Proof of main theorem}\label{yu-pro-12-23-1}

This section is devoted to prove Theorem \ref{yu-theorem-12-3-1}.

\begin{proof}[Proof of Theorem \ref{yu-theorem-12-3-1}]
    Arbitrarily fix  $(\varepsilon,T)\in[0,\alpha)\times (0,+\infty)$ satisfying \eqref{yu-mmmm-1}.
     Let $\mathcal{V}_{\alpha,\varepsilon,T}(\cdot)$ be the $C_0$-group on $H'$
    generated by $-A^*-P_{\alpha,\varepsilon,T}$, where $P_{\alpha,\varepsilon,T}$ is given by \eqref{yu0-bb-1}. Define
\begin{equation}\label{yu-11-24-24}
    \mathcal{S}_{\alpha,\varepsilon,T}(t):=\Pi_{\alpha,\varepsilon,T}
    \mathcal{V}_{\alpha,\varepsilon,T}(t)\Pi^{-1}_{\alpha,\varepsilon,T},\;\;t\in \mathbb{R}.
\end{equation}
    (The invertibility of  $\Pi_{\alpha,\varepsilon,T}$ is ensured by $(iii)$ of Lemma \ref{yu-lemma-12-3-2}.)
    The rest of the proof is
    organized in several steps.
\vskip 5pt
   \noindent \emph{Step 1. We have the following conclusions:
\begin{enumerate}
  \item [$(a_1)$] The family $\{\mathcal{S}_{\alpha,\varepsilon,T}(t)\}_{t\in \mathbb{R}}$, given by
    \eqref{yu-11-24-24}, is a $C_0$-group on $H$;
  \item [$(a_2)$] The generator of $\mathcal{S}_{\alpha,\varepsilon,T}(\cdot)$ is as:
  $\mathcal{A}_{\alpha,\varepsilon,T}:=\Pi_{\alpha,\varepsilon,T}(-A^*-P_{\alpha,\varepsilon,T}
  )\Pi^{-1}_{\alpha,\varepsilon,T}$,
  with its domain $D(\mathcal{A}_{\alpha,\varepsilon,T})=\Pi_{\alpha,\varepsilon,T}[D(A^*)]$;
  \item [$(a_3)$] For any $t\in \mathbb{R}^+$, $x\in D(\mathcal{A}_{\alpha,\varepsilon,T})$ and $\varphi\in D(A^*)$,
\begin{eqnarray*}\label{yu-11-24-25}
    \langle \mathcal{S}_{\alpha,\varepsilon,T}(t)x,\varphi\rangle_{H,H'}
    =\langle x,S^*(t)\varphi\rangle_{H,H'}
    -TC_1(\alpha)e^{\alpha T}\int_0^t\langle J_1B^*\Pi_{\alpha,\varepsilon,T}^{-1}
    \mathcal{S}_{\alpha,\varepsilon,T}(s)x,B^*S^*(t-s)\psi\rangle_{U,U'}ds.
\end{eqnarray*}
\end{enumerate}}
\par
    These can be proved by  very similar methods used in the proof of  \cite[Theorem 3.2]{Vest}. We omit the proofs.
    % First, since $\mathcal{V}_{\alpha,\varepsilon,T}(\cdot)$ is a $C_0$-group, the claim $(a)$ is obvious. Second, since $\mathcal{S}_{\alpha,\varepsilon,T}(\cdot)$ is the similar semigroup (for its definition, we can refer to \cite[Section 5.10 on Page 43]{Engel-Nagel})
%    of $\mathcal{V}_{\alpha,\varepsilon,T}(\cdot)$, the claim $(b)$ comes from
%    \cite[Section 2.1 on Page 59]{Engel-Nagel}. Third, by using (\ref{yu-11-23-2}) with replacing $\varphi$ and $\psi$ by
%    $\Pi_{\alpha,\varepsilon,T}^{-1}x_0$ and $S^*(t)\varphi$, respectively, we can obtain
%    the equality (\ref{yu-11-24-25}). Here, we note that
%    once $x_0\in D(\mathcal{A}_{\alpha,\varepsilon,T})$, then $\Pi_{\alpha,\varepsilon,T}^{-1}x_0\in D(A^*)$.
\vskip 5pt
   \noindent \emph{Step 2. We have
\begin{equation*}\label{yu-11-26-10}
    (\widetilde{A}-T C_1(\alpha)e^{\alpha T}BJ_1B^*\Pi_{\alpha,\varepsilon,T}^{-1})x
    =\mathcal{A}_{\alpha,\varepsilon,T}x\;\;\mbox{for all}\;\;x\in D(\mathcal{A}_{\alpha,\varepsilon,T}),
\end{equation*}
    where $\widetilde{A}\in \mathcal{L}(H;D(A^*)')$ is the unique extension of $A$,  defined by (\ref{1.3,202211}).}
\par
    The very similar result  has been proved in \cite[Theorem 3.3]{Vest} by using the conclusions in Step 1. Thus, we omit its proof.
\vskip 5pt
   \noindent \emph{Step 3. We prove that for any $\varphi,\psi\in D(A^*)$ and  $t\in\mathbb{R}$,
\begin{eqnarray}\label{yu-11-28-2}
    \langle \Pi_{\alpha,\varepsilon,T}\varphi,\psi\rangle_{H,H'}
     &=& TC_1(\alpha)e^{\alpha T} \int_0^t\langle J_1B^*\mathcal{V}_{\alpha,\varepsilon,T}
    (s)\varphi, B^*\mathcal{V}_{\alpha,\varepsilon,T}(s)\psi\rangle_{U,U'}ds
    \\
    &\;&+\langle \Pi_{\alpha,\varepsilon,T}\mathcal{V}_{\alpha,\varepsilon,T}(t)\varphi,
    \mathcal{V}_{\alpha,\varepsilon,T}(t)\psi\rangle_{H,H'}
       +\int_0^t\langle \mathcal{V}_{\alpha,\varepsilon,T}(s)\varphi,\widehat{Q_{\alpha,\varepsilon,T}}
\mathcal{V}_{\alpha,\varepsilon,T}(s)\psi\rangle_{H',H}ds,\nonumber
\end{eqnarray}
    where $\widehat{Q_{\alpha,\varepsilon,T}}$ is defined by
\begin{equation}\label{yu-11-23-3}
\widehat{Q_{\alpha,\varepsilon,T}}:=
\Pi_{\alpha,\varepsilon,T}P_{\alpha,\varepsilon,T}.
\end{equation}
    }

     First of all, it follows from (\ref{yu-11-23-1}), the assumption $(H_3)$ and Lemma \ref{yu-lemma-11-23-2}
     that
     the first term on the right hand of (\ref{yu-11-28-2}) makes sense.

     We now arbitrarily fix  $\varphi,\psi\in D(A^*)$ and $t\in\mathbb{R}$. Then by Lemma \ref{yu-lemma-11-23-1} and (\ref{yu-11-23-1}),  we have
\begin{eqnarray}\label{yu-11-28-3}
    \langle \Pi_{\alpha,\varepsilon,T}\varphi,\psi\rangle_{H,H'}
    &=&\langle \Pi_{\alpha,\varepsilon,T}\mathcal{V}_{\alpha,\varepsilon,T}(t)\varphi,
    \mathcal{V}_{\alpha,\varepsilon,T}(t)\psi\rangle_{H,H'}\nonumber\\
    &&+TC_1(\alpha)e^{\alpha T}
    \int_0^t\langle J_1B^*\mathcal{V}_{\alpha,\varepsilon,T}(s)\varphi,
    B^*\mathcal{V}_{\alpha,\varepsilon,T}(s)\psi\rangle_{U,U'}ds
    +\mathcal{W}_1(t)+\mathcal{W}_2(t),
\end{eqnarray}
    where
\begin{equation*}\label{yu-11-28-4}
\begin{cases}
    \mathcal{W}_1(t):=\left\langle \Pi_{\alpha,\varepsilon,T}\mathcal{V}_{\alpha,\varepsilon,T}(t)\varphi,
    \int_0^tS^*(s-t)P_{\alpha,\varepsilon,T}
    \mathcal{V}_{\alpha,\varepsilon,T}(s)\psi ds\right\rangle_{H,H'},\\
    \mathcal{W}_2(t):=TC_1(\alpha)e^{\alpha T}
    \int_0^t\left\langle J_1B^*\mathcal{V}_{\alpha,\varepsilon,T}(s)\varphi,
    B^*\int_0^sS^*(\sigma-s)P_{\alpha,\varepsilon,T}\mathcal{V}_{\alpha,\varepsilon,T}
    (\sigma)\psi d\sigma\right\rangle_{U,U'}ds.
\end{cases}
\end{equation*}
    (It follows  by the assumption $(H_3)$, Lemma \ref{yu-lemma-11-23-2} and (\ref{yu-11-23-1})
     that the term $\mathcal{W}_2(t)$ makes sense.)

    Next, we will show
        \begin{equation}\label{yu-11-28-13}
    \mathcal{W}_1(t)=\int_0^t\langle
\mathcal{V}_{\alpha,\varepsilon,T}(s)\varphi,
    \widehat{Q_{\alpha,\varepsilon,T}}\mathcal{V}_{\alpha,\varepsilon,T}(s)
    \psi\rangle_{H',H}ds-\mathcal{W}_2(t).
\end{equation}
    When this is done, \eqref{yu-11-28-2} follows from \eqref{yu-11-28-3}
    and \eqref{yu-11-28-13} at once.

    To show \eqref{yu-11-28-13}, we let $n^*\in\mathbb{N}^+$  so that
    when $n\geq n^*$,
    $nI-A^*$ is invertible.
    We define, for each   $n\geq n^*$,
\begin{equation}\label{yu-kkk-1}
    \mathcal{R}_n:=n(nI-A^*)^{-1}
\end{equation}
and
\begin{equation}\label{yu-11-28-5w}
    \mathcal{K}_n(t):=\Big\langle \Pi_{\alpha,\varepsilon,T}\mathcal{V}_{\alpha,\varepsilon,T}(t)\varphi,
    \int_0^tS^*(s-t)\mathcal{R}_nP_{\alpha,\varepsilon,T}
    \mathcal{V}_{\alpha,\varepsilon,T}(s)\psi ds\Big\rangle_{H,H'}.
\end{equation}
By  \cite[Chapter 1, Theorem 6.3]{Pazy}, we can find  two positive numbers $c_1$ and $c_2$ such that
\begin{equation}\label{yu-11-24-2}
    \|\mathcal{R}_n\|_{\mathcal{L}(H')}\leq \frac{nc_1}{n-c_2}\;\;\mbox{for each}\;\;n\geq n^*,
\end{equation}
   while by \cite[Chapter 1, Lemma 3.2]{Pazy}, we see that when $x\in H'$,
\begin{equation}\label{yu-kkkk-2}
    \mathcal{R}_nx\in D(A^*)\;\;\mbox{and}\;\;
    \mathcal{R}_nx\to x\;\;\mbox{in}\;\; H'\;\;\mbox{as}\;\;
     n\to +\infty.
\end{equation}
    Then by \eqref{yu-11-24-2} and (\ref{yu-kkkk-2}), we can apply
    the dominated convergence theorem in \eqref{yu-11-28-5w} to get
\begin{equation}\label{yu-11-28-6}
    \lim_{n\to+\infty}\mathcal{K}_n(t)=\mathcal{W}_1(t).
\end{equation}
(Here, we used the definition of $\mathcal{W}_1(t)$.)
    Meanwhile, by (\ref{yu-11-23-2}) in Lemma \ref{yu-lemma-11-23-1} (where we replace $\varphi$ and $\psi$ by
    $\mathcal{V}_{\alpha,\varepsilon,T}(s)\varphi$ and $\mathcal{R}_nP_{\alpha,\varepsilon,T}
    \mathcal{V}_{\alpha,\varepsilon,T}(s)\psi$, respectively), we have
\begin{eqnarray}\label{yu-11-28-7}
    \mathcal{K}_n(t)&=&\int_0^t\langle
    \Pi_{\alpha,\varepsilon,T}\mathcal{V}_{\alpha,\varepsilon,T}(t-s)
    \mathcal{V}_{\alpha,\varepsilon,T}(s)\varphi, S^*(s-t)\mathcal{R}_nP_{\alpha,\varepsilon,T}
    \mathcal{V}_{\alpha,\varepsilon,T}(s)\psi\rangle_{H,H'}ds\nonumber\\
    &=&\mathcal{K}_{n,1}(t)+\mathcal{K}_{n,2}(t),
\end{eqnarray}
    where
\begin{equation*}\label{yu-11-28-8}
\begin{cases}
    \mathcal{K}_{n,1}(t):=
    \int_0^t\langle \Pi_{\alpha,\varepsilon,T}\mathcal{V}_{\alpha,\varepsilon,T}(s)\varphi,
    \mathcal{R}_nP_{\alpha,\varepsilon,T}
    \mathcal{V}_{\alpha,\varepsilon,T}(s)\psi\rangle_{H,H'}ds,\\
    \mathcal{K}_{n,2}(t):=-TC_1(\alpha)e^{\alpha T}\int_0^t\int_0^{t-s}
    \langle J_1B^*\mathcal{V}_{\alpha,\varepsilon,T}(\sigma+s)\varphi,
    B^*S^*(-\sigma)\mathcal{R}_nP_{\alpha,\varepsilon,T}
    \mathcal{V}_{\alpha,\varepsilon,T}(s)\psi\rangle_{U,U'}d\sigma ds.
\end{cases}
\end{equation*}
   With respect to $\mathcal{K}_{n,1}(t)$,  we obtain, from \eqref{yu-11-24-2}, (\ref{yu-kkkk-2}), (\ref{yu-11-23-3}) and the dominated convergence theorem,   that
\begin{eqnarray}\label{yu-11-28-9}
    \lim_{n\to+\infty} \mathcal{K}_{n,1}(t)=
    \int_0^t\langle \mathcal{V}_{\alpha,\varepsilon,T}(s)\varphi,
    \widehat{Q_{\alpha,\varepsilon,T}}\mathcal{V}_{\alpha,\varepsilon,T}(s)\psi\rangle_{H',H}ds.
\end{eqnarray}
 With respect to $\mathcal{K}_{n,2}(t)$, we will claim
\begin{equation}\label{yu-11-28-12w}
    \lim_{n\to +\infty}\mathcal{K}_{n,2}(t)=-\mathcal{W}_2(t).
\end{equation}
To this end, it suffices to show
\begin{eqnarray}\label{yu-11-28-10}
    &\;&\int_0^t\int_0^{t-s}
    \langle J_1B^*\mathcal{V}_{\alpha,\varepsilon,T}(\sigma+s)\varphi,
    B^*S^*(-\sigma)\mathcal{R}_nP_{\alpha,\varepsilon,T}
    \mathcal{V}_{\alpha,\varepsilon,T}(s)\psi\rangle_{U,U'}d\sigma ds\nonumber\\
       &=&\int_0^t\langle J_1B^*\mathcal{V}_{\alpha,\varepsilon,T}(\gamma)\varphi,
    B^*\int_0^\gamma S^*(s-\gamma)\mathcal{R}_n P_{\alpha,\varepsilon,T}\mathcal{V}_{\alpha,\varepsilon,T}(s)
    \psi ds\rangle_{U,U'}d\gamma,
\end{eqnarray}
    and
\begin{eqnarray}\label{12-18-27w}
    &&\lim_{n\rightarrow+\infty}\int_{-|t|}^{|t|}\Big\|B^*\int_0^\gamma S^*(s-\gamma)(\mathcal{R}_n-I)P_{\alpha,\varepsilon,T}
    \mathcal{V}_{\alpha,\varepsilon,T}(s)\psi ds\Big\|^2_{U'}d\gamma=0.
\end{eqnarray}
When these have been done, \eqref{yu-11-28-12w} follows from (\ref{yu-11-28-10}),
  \eqref{12-18-27w} and the definitions of $\mathcal{K}_{n,2}(t)$ and $\mathcal{W}_2(t)$
  at once.

To show (\ref{yu-11-28-10}), we first  notice that
by the note $(ii)$ in Remark \ref{yu-remark-12-2-1},
\begin{equation*}
    B^*\mathcal{R}_n=E^*(A+\lambda I)^*(n(nI-A^*)^{-1})=
    nE^*+(n^2+n\bar{\lambda})E^*(nI-A^*)^{-1},\;\;\mbox{when}\;\;n\geq n^*,
\end{equation*}
   (Here $\bar{\lambda}$ is the conjugate of $\lambda$.) which leads to
   \begin{equation}\label{yu-11-24-16w}
    B^*\mathcal{R}_n\in \mathcal{L}(H')\;\;\mbox{for all}\;\; n\geq n^*.
\end{equation}
   Next, since  $\mathcal{R}_n=n\int_0^{+\infty}e^{-nt}S^*(t)dt$ (see the proof Theorem 3.1 in \cite[Chapter 1]{Pazy})), we have $\mathcal{R}_nS^*(\cdot)=S^*(\cdot) \mathcal{R}_n$.
   From this,  \eqref{yu-11-24-16w},
     Lemma \ref{yu-lemma-11-23-2} and
     \cite[Lemma 11.45]{Aliprantis-Border}, we find
\begin{eqnarray*}
    &\;&\int_0^t\int_0^{t-s}
    \langle J_1B^*\mathcal{V}_{\alpha,\varepsilon,T}(\sigma+s)\varphi,
    B^*S^*(-\sigma)\mathcal{R}_nP_{\alpha,\varepsilon,T}
    \mathcal{V}_{\alpha,\varepsilon,T}(s)\psi\rangle_{U,U'}d\sigma ds\nonumber\\
    &=&\int_0^t\int_s^t\langle J_1B^*\mathcal{V}_{\alpha,\varepsilon,T}(\gamma)\varphi,
    B^*S^*(s-\gamma)\mathcal{R}_nP_{\alpha,\varepsilon,T}\mathcal{V}_{\alpha,\varepsilon,T}(s)
    \psi\rangle_{U,U'}d\gamma ds\nonumber\\
    &=&\int_0^t\int_0^\gamma\langle J_1B^*\mathcal{V}_{\alpha,\varepsilon,T}(\gamma)\varphi,
    B^*\mathcal{R}_nS^*(s-\gamma)P_{\alpha,\varepsilon,T}\mathcal{V}_{\alpha,\varepsilon,T}(s)
    \psi\rangle_{U,U'}dsd\gamma\nonumber\\
    &=&\int_0^t\langle J_1B^*\mathcal{V}_{\alpha,\varepsilon,T}(\gamma)\varphi,
    B^*\mathcal{R}_n\int_0^\gamma S^*(s-\gamma)P_{\alpha,\varepsilon,T}\mathcal{V}_{\alpha,\varepsilon,T}(s)
    \psi ds\rangle_{U,U'}d\gamma \nonumber\\
    &=&\int_0^t\langle J_1B^*\mathcal{V}_{\alpha,\varepsilon,T}(\gamma)\varphi,
    B^*\int_0^\gamma S^*(s-\gamma)\mathcal{R}_n P_{\alpha,\varepsilon,T}\mathcal{V}_{\alpha,\varepsilon,T}(s)
    \psi ds\rangle_{U,U'}d\gamma,
\end{eqnarray*}
which leads to (\ref{yu-11-28-10}).

To show \eqref{12-18-27w}, we let
     $z_n(\gamma):=\int_0^\gamma S^*(s-\gamma)(\mathcal{R}_n-I)P_{\alpha,\varepsilon,T}\mathcal{V}_{\alpha,\varepsilon,T}(s)\psi
    ds$,
    then, by Lemma \ref{yu-lemma-11-23-2}, we can find $C(|t|)>0$ such that
\begin{eqnarray*}
    \int_{-|t|}^{|t|}\|B^*z_n(\gamma)\|_{U'}^2d\gamma
    &\leq&C(|t|)\biggl(\|(\mathcal{R}_n-I)P_{\alpha,\varepsilon,T}\psi\|^2_{H'}
    +\int_{-|t|}^{|t|}\|(\mathcal{R}_n-I)P_{\alpha,\varepsilon,T}\mathcal{V}_{\alpha,\varepsilon,T}(s)\psi\|^2_{H'}ds\nonumber\\
    &\;&+\int_{-|t|}^{|t|}
    \|(\mathcal{R}_n-I)P_{\alpha,\varepsilon,T}\mathcal{V}_{\alpha,\varepsilon,T}(s)
    \Delta_{\alpha,\varepsilon,T}\psi\|^2_{H'}\biggl)ds.
\end{eqnarray*}
     This, together with \eqref{yu-11-24-2}, (\ref{yu-kkkk-2}) and the dominated convergence theorem, leads to \eqref{12-18-27w}.

    Finally, \eqref{yu-11-28-13} follows from \eqref{yu-11-28-12w},  (\ref{yu-11-28-6}), (\ref{yu-11-28-7}) and (\ref{yu-11-28-9}) at once. This ends the proof of Step 3.

\vskip 5pt
   \noindent \emph{Step 4. We show that when $x,y\in D(\mathcal{A}_{\alpha,\varepsilon,T})$ and  $t\in\mathbb{R}$,
\begin{eqnarray}\label{yu-11-28-14}
    &\;&\langle x,\Pi_{\alpha,\varepsilon,T}^{-1}y\rangle_{H,H'}\nonumber\\
    &=&
    \langle \mathcal{S}_{\alpha,\varepsilon,T}(t)x,\Pi_{\alpha,\varepsilon,T}^{-1}
    \mathcal{S}_{\alpha,\varepsilon,T}(t)y\rangle_{H,H'}
    +\int_0^t\langle\Pi_{\alpha,\varepsilon,T}^{-1}
    \mathcal{S}_{\alpha,\varepsilon,T}(s)x,\widehat{Q_{\alpha,\varepsilon,T}}
\Pi_{\alpha,\varepsilon,T}^{-1}\mathcal{S}_{\alpha,\varepsilon,
    T}(s)y\rangle_{H',H}ds\nonumber\\
    &\;&+TC_1(\alpha)e^{\alpha T}\int_0^t
    \langle J_1B^*\Pi_{\alpha,\varepsilon,T}^{-1}\mathcal{S}_{\alpha,\varepsilon,T}(s)x,
    B^*\Pi_{\alpha,\varepsilon,T}^{-1}\mathcal{S}_{\alpha,\varepsilon,T}(s)y\rangle_{U,U'}ds.
\end{eqnarray}
}

First of all, the third term on the right hand of (\ref{yu-11-28-14}) makes sense. The reason is as:
it follows  by $(a_2)$ in Step 1 that when
     $z\in D(\mathcal{A}_{\alpha,\varepsilon,T})$, we have $\Pi_{\alpha,\varepsilon,T}^{-1}z\in D(A^*)$. Thus,
     it follows from (\ref{yu-11-24-24}), (\ref{yu-11-23-1}) and Lemma \ref{yu-lemma-11-23-2} that $B^*\Pi_{\alpha,\varepsilon,T}^{-1}\mathcal{S}_{\alpha,\varepsilon,T}(\cdot)
    z=B^*\mathcal{V}_{\alpha,\varepsilon,T}
    (\cdot)\Pi_{\alpha,\varepsilon,T}^{-1}z\in L^2_{loc}(\mathbb{R};U')$.

    Next, we arbitrarily fix  $x, y\in D(\mathcal{A}_{\alpha,\varepsilon,T})$ and $t\in\mathbb{R}$.
    Then by the conclusion $(a_2)$ in Step 1, we find
    $\Pi_{\alpha,\varepsilon,T}^{-1}x, \Pi_{\alpha,\varepsilon,T}^{-1}y\in D(A^*)$.
    This, along with (\ref{yu-11-28-2}) (where $\varphi,\psi$ are replaced by $\Pi_{\alpha,\varepsilon,T}^{-1}x, \Pi_{\alpha,\varepsilon,T}^{-1}y$, respectively)
    and (\ref{yu-11-24-24}), yields (\ref{yu-11-28-14}).
\vskip 5pt
  \noindent  \emph{Step 5. Let
    \begin{equation}\label{12-192.29w}
    \mathcal{S}(\cdot):=\mathcal{S}_{\alpha,\varepsilon,T}(\cdot)\;\;\mbox{and}\;\;K:=-TC_1(\alpha)e^{\alpha T}J_1B^*\Pi_{\alpha,\varepsilon,T}^{-1}.
    \end{equation}
    We show that $K$ is a feedback law stabilizing \eqref{yu-12-2-b-1} with the decay rate $\frac{1}{2}(\alpha-\varepsilon)$.
    }

    First of all, by the conclusions $(a_1)$ and $(a_2)$ in Step 1, we see that $\mathcal{S}(\cdot)$ is a $C_0$-group with the generator:
     \begin{equation}\label{12-192.30w}
     \mathcal{A}:=\mathcal{A}_{\alpha,\varepsilon,T}=\Pi_{\alpha,\varepsilon,T}
  (-A^*-P_{\alpha,\varepsilon,T})\Pi^{-1}_{\alpha,\varepsilon,T},\;\;\mbox{with its domain}\; D(\mathcal{A})=\Pi_{\alpha,\varepsilon,T}[D(A^*)].
     \end{equation}
   It follows by \eqref{12-192.30w},  $(ii)$ in Remark \ref{yu-remark-12-2-1} and the conclusion $(a_2)$ in Step 1
   that for each  $x\in D(\mathcal{A})$,  $\Pi_{\alpha,\varepsilon,T}^{-1}x\in D(A^*)$ and
   that
\begin{eqnarray*}
    \|B^*\Pi_{\alpha,\varepsilon,T}^{-1}x\|_{U'}&=&\|E^*(A^*+P_{\alpha,\varepsilon,T})
    \Pi_{\alpha,\varepsilon,T}^{-1}x+E^*(\bar{\lambda}I-P_{\alpha,\varepsilon,T})
    \Pi_{\alpha,\varepsilon,T}^{-1}x\|_{U'}\nonumber\\
    &=&\|E^*\Pi_{\alpha,\varepsilon,T}^{-1}\mathcal{A}x
    +E^*(\bar{\lambda}I-P_{\alpha,\varepsilon,T})
    \Pi_{\alpha,\varepsilon,T}^{-1}x\|_{U'}\nonumber\\
    &\leq& \left(\|E^*\Pi_{\alpha,\varepsilon,T}^{-1}\|_{\mathcal{L}(H;U')}
    +\|E^*(\bar{\lambda}I-P_{\alpha,\varepsilon,T})
    \Pi_{\alpha,\varepsilon,T}^{-1}\|_{\mathcal{L}(H;U')}\right)\|x\|_{D(\mathcal{A})}.
\end{eqnarray*}
    These, along with \eqref{12-192.29w} and the conjugate-linearity of $J_1$ and $\Pi_{\alpha,\varepsilon,T}$, yields $K\in\mathcal{L}(D(\mathcal{A});U)$.

    Next, we will show that
     $\mathcal{S}(\cdot)$ (as well as $\mathcal{A}$) and $K$ verify the conditions $(i)$,
    $(ii)$ and $(iii)$ in  Definition \ref{yu-definition-11-17-1} (with $\omega=\frac{1}{2}(\alpha-\varepsilon)$) one by one.
\par
\vskip 5pt
\noindent  \emph{Sub-step 5.1.  We prove $(i)$ in  Definition \ref{yu-definition-11-17-1} with $\omega=\frac{1}{2}(\alpha-\varepsilon)$}.

We first claim
\begin{equation}\label{yu-11-28-19}
    \langle\mathcal{S}(t)x,\Pi_{\alpha,\varepsilon,T}^{-1}\mathcal{S}(t)x\rangle_{H,H'}
    \leq  e^{-(\alpha-\varepsilon)t}\langle x, \Pi^{-1}_{\alpha,\varepsilon,T}x\rangle_{H,H'}\;\;\mbox{for all}\;\;x\in D(\mathcal{A}),\; t\in\mathbb{R}^+.
\end{equation}
To this end, we arbitrarily fix $x\in D(\mathcal{A})$, $t$ and $\sigma$ with $t\geq \sigma\geq 0$.
    Then by (\ref{yu-11-28-14}), we have
\begin{eqnarray}\label{yu-11-28-15}
    \langle \mathcal{S}(\sigma)x,\Pi_{\alpha,\varepsilon,T}^{-1}\mathcal{S}(\sigma)x\rangle_{H,H'}
    &=& \langle \mathcal{S}(t)x,\Pi_{\alpha,\varepsilon,T}^{-1}
    \mathcal{S}(t)x\rangle_{H,H'}
    +\int_\sigma^t\langle \Pi_{\alpha,\varepsilon,T}^{-1}\mathcal{S}(s)x,
    \widehat{Q_{\alpha,\varepsilon,T}}\Pi_{\alpha,\varepsilon,T}^{-1}
    \mathcal{S}(s)x\rangle_{H',H}ds\nonumber\\
    &\;&+TC_1(\alpha)e^{\alpha T}\int_\sigma^t
    \langle J_1B^*\Pi_{\alpha,\varepsilon,T}^{-1}\mathcal{S}(s)x,
    B^*\Pi_{\alpha,\varepsilon,T}^{-1}\mathcal{S}(s)x\rangle_{U,U'}ds.
\end{eqnarray}
    Meanwhile, by (\ref{yu0-bb-1}) and (\ref{yu-11-23-3}) (the definitions of $P_{\alpha,\varepsilon,T}$ and $\widehat{Q_{\alpha,\varepsilon,T}}$), we see
\begin{equation}\label{yu-11-28-16w}
    \widehat{Q_{\alpha,\varepsilon,T}}\Pi_{\alpha,\varepsilon,T}^{-1}=
    (\alpha-\varepsilon)I+Q_{\alpha,\varepsilon,T}\Pi_{\alpha,\varepsilon,T}^{-1},
\end{equation}
    where $Q_{\alpha,\varepsilon,T}$ is given by (\ref{yu-11-13-1-b})
    and is non-negative (which follows from $(ii)$ in Proposition \ref{yu-proposition-11-12-1}, since $(\varepsilon,T)$ verifies (\ref{yu-mmmm-1})). Now, by \eqref{yu-11-28-16w} and the non-negativity of
    $Q_{\alpha,\varepsilon,T}$, we find
\begin{eqnarray}\label{yu-11-28-17}
    &\;&\int_\sigma^t\langle \Pi_{\alpha,\varepsilon,T}^{-1}\mathcal{S}(s)x,
    \widehat{Q_{\alpha,\varepsilon,T}}
    \Pi_{\alpha,\varepsilon,T}^{-1}\mathcal{S}(s)x\rangle_{H',H}ds \nonumber\\
    &=&(\alpha-\varepsilon)\int_{\sigma}^t\langle \mathcal{S}(s)x,\Pi_{\alpha,\varepsilon,T}^{-1}
    \mathcal{S}(s)x\rangle_{H,H'}ds+\int_{\sigma}^t\langle Q_{\alpha,\varepsilon,T}\Pi_{\alpha,\varepsilon,T}^{-1}\mathcal{S}(s)x,
    \Pi_{\alpha,\varepsilon,T}^{-1}
    \mathcal{S}(s)x\rangle_{H,H'}ds\nonumber\\
    &\geq& (\alpha-\varepsilon)\int_{\sigma}^t\langle \mathcal{S}(s)x,\Pi_{\alpha,\varepsilon,T}^{-1}
    \mathcal{S}(s)x\rangle_{H,H'}ds.
\end{eqnarray}
    From \eqref{yu-11-28-17} and (\ref{yu-11-28-15}), it follows that
\begin{eqnarray*}\label{yu-11-28-18}
    \langle \mathcal{S}(\sigma)x,\Pi_{\alpha,\varepsilon,T}^{-1}\mathcal{S}(\sigma)x\rangle_{H,H'}
    \geq \langle \mathcal{S}(t)x,\Pi_{\alpha,\varepsilon,T}^{-1}
    \mathcal{S}(t)x\rangle_{H,H'}
    +(\alpha-\varepsilon)\int_{\sigma}^t\langle \mathcal{S}(s)x,\Pi_{\alpha,\varepsilon,T}^{-1}
    \mathcal{S}(s)x\rangle_{H,H'}ds.
\end{eqnarray*}
    Since $x\in D(\mathcal{A})$ and $t\geq \sigma\geq 0$ were arbitrarily taken, we can apply  the Gronwall inequality (see \cite[Lemma 3.2]{Komornik-1997}) in the above inequality to get \eqref{yu-11-28-19}.

    We next claim that there exists  $C(\alpha,\varepsilon,T)>0$ such that
    \begin{equation}\label{yu-11-28-21}
    T(C(\alpha,\varepsilon,T))^{-2}\|y\|^2_{H}
    \leq \langle y,\Pi_{\alpha,\varepsilon,T}^{-1}y\rangle_{H,H'}
    \leq T^{-1}\|y\|_{H}^2\;\;\mbox{for any}\;\;y\in H.
\end{equation}
Indeed, according to  the second inequality in (\ref{yu-12-3-10}) and the boundedness of $\Pi_{\alpha,\varepsilon,T}$ (see Lemma \ref{yu-lemma-12-3-2}), there exists  $C(\alpha,\varepsilon,T)>0$ such that for all $\varphi\in H'$,
\begin{equation*}\label{yu-11-28-20}
T\|\varphi\|_{H'}^2\leq \langle\Pi_{\alpha,\varepsilon, T}\varphi,\varphi\rangle_{H,H'}
\;\;\mbox{and}\;\;
\|\Pi_{\alpha,\varepsilon, T}\varphi\|_{H}\leq C(\alpha,\varepsilon,T)\|\varphi\|_{H'}.
\end{equation*}
    These lead to \eqref{yu-11-28-21}.

    Now it follows from  \eqref{yu-11-28-21} and  (\ref{yu-11-28-19}) that
\begin{equation*}\label{yu-11-28-22}
    \|\mathcal{S}(t)x\|_H\leq \left(\frac{C(\alpha,\varepsilon,T)}{T}\right)^2
    e^{-(\alpha-\varepsilon)t}\|x\|^2_H\;\;\mbox{for any}\;\;t\in\mathbb{R}^+,\; x\in D(\mathcal{A}).
\end{equation*}
    This, together with the density of $D(\mathcal{A})$ in $H$, shows
\begin{equation*}\label{yu-11-28-23}
    \|\mathcal{S}(t)\|_{\mathcal{L}(H)}\leq \frac{C(\alpha,\varepsilon,T)}{T}e^{-\frac{1}{2}(\alpha-\varepsilon)t}\;\;\mbox{for any}\;\;t\in\mathbb{R}^+,
\end{equation*}
   i.e.,  $(i)$ in  Definition \ref{yu-definition-11-17-1} with $\omega=\frac{1}{2}(\alpha-\varepsilon)$
   is true.
\par
   \vskip 5pt
\noindent  \emph{Sub-step 5.2. We prove $(ii)$ in  Definition \ref{yu-definition-11-17-1} with $\omega=\frac{1}{2}(\alpha-\varepsilon)$. }

This follows from  \eqref{12-192.29w} and the conclusion in Step 2 at once.

\par
\vskip 5pt
\noindent  \emph{Sub-step 5.3. We prove $(iii)$ in  Definition \ref{yu-definition-11-17-1} with $\omega=\frac{1}{2}(\alpha-\varepsilon)$. }

In the case that  $C_1(\alpha)=0$, we see from \eqref{12-192.29w}
that $K=0$, thus  $(iii)$ holds for this case.

In the case that  $C_1(\alpha)\neq 0$,
it follows by  (\ref{yu-11-28-15}), (\ref{yu-11-28-17}) (with $\sigma=0$) and \eqref{12-192.29w} that
when $x\in D(\mathcal{A})$,
\begin{equation*}\label{yu-11-28-24}
    \langle x,\Pi_{\alpha,\varepsilon,T}^{-1}x\rangle_{H,H'}
    \geq \langle \mathcal{S}(t)x,\Pi^{-1}_{\alpha,\varepsilon,T}\mathcal{S}(t)x\rangle_{H,H'}
    +(TC_1(\alpha)e^{\alpha T})^{-1}\int_0^t\|K\mathcal{S}(s)x\|^2_{U'}ds
    \;\;\mbox{for each}\;\;t\in \mathbb{R}^+.
\end{equation*}
    Letting $t\to +\infty$ in the above, using $(i)$ and
      (\ref{yu-11-28-21}), we see
\begin{equation*}\label{yu-11-28-25}
    \int_0^{+\infty}\|K\mathcal{S}(s)x\|^2_{U'}ds\leq
    C_1(\alpha)e^{\alpha T}\|x\|_{H}^2
    \;\;\mbox{for any}\;\;x\in D(\mathcal{A}),
\end{equation*}
   which leads to $(iii)$ for this case.
\par
    In summary, $\mathcal{S}(\cdot)$ (as well as $\mathcal{A}$) and $K$ verify the conditions $(i)$, $(ii)$ and $(iii)$  in Definition \ref{yu-definition-11-17-1}
    with $\omega=\frac{1}{2}(\alpha-\varepsilon)$.

    \par
\vskip 5pt
\noindent  \emph{Step 6. We finish the proof. }

Arbitrarily fix  $T$ satisfying (\ref{1.7,12-27w}). Let $\widehat{\varepsilon}:=T^{-1}\ln [C_2(\alpha)]$.
Then one can easily check that $(\widehat{\varepsilon}, T)$ satisfies
  (\ref{yu-mmmm-1}). Thus by  the conclusions in Step 5, we complete
     the proof of Theorem \ref{yu-theorem-12-3-1}.
\end{proof}

\vskip 5pt
    Our proof of Theorem \ref{yu-theorem-12-3-1} shows, indeed,  the following
         more general result:
\begin{theorem}\label{yu-theorem-12-16-1}
    Assume that $(H_1)$-$(H_4)$ are true.
    Then for each  pair $(\varepsilon,T)\in[0,\alpha)\times (0,+\infty)$ verifying
    (\ref{yu-mmmm-1}), the following operator (from $\Pi_{\alpha,\varepsilon,T}[D(A^*)]$ to $U$) is a feedback law stabilizing the system (\ref{yu-12-2-b-1}) with the decay rate $\frac{1}{2}(\alpha-\varepsilon)$:
\begin{equation*}\label{yu-12-3-5}
     K_{\varepsilon,T}:=-TC_1(\alpha)e^{\alpha T}J_1B^*\Pi_{\alpha,\varepsilon,T}^{-1},
\end{equation*}
    where
    $\Pi_{\alpha,\varepsilon,T}$ is defined by (\ref{yu-11-12-3}).
\end{theorem}

\begin{remark}
 $(i)$ Several facts are given. First,
it follows from \eqref{yu-11-28-14} that the operator
$\Pi^{-1}_{\alpha,\varepsilon,T}$, with $(\varepsilon,T)\in [0,\alpha)\times(0,+\infty)$,
 satisfies the following Riccati equation:
 \begin{eqnarray}\label{yu-12-11-1}
    &\;&\langle A^*\mathcal{P}x,y\rangle_{H',H}+\langle x, A^*\mathcal{P}y\rangle_{H,H'}
    -TC_1(\alpha)e^{\alpha T}\langle JB^*\mathcal{P}x,B^*\mathcal{P}y\rangle_{H,H'}\nonumber\\
    &=&-\left\langle \left((\alpha-\varepsilon)\Pi_{\alpha,\varepsilon,T}^{-1}
    +\Pi_{\alpha,\varepsilon,T}^{-1}Q_{\alpha,\varepsilon,T}
    \Pi_{\alpha,\varepsilon,T}^{-1}\right)
    x,y\right\rangle_{H',H},\;\;x,y\in \Pi_{\alpha,\varepsilon,T}[D(A^*)],
\end{eqnarray}
  where
     $Q_{\alpha,\varepsilon,T}
    :=\Lambda_{\alpha,\varepsilon,T}(T)-C_2(\alpha) J_2$ is a conjugate-linear and bounded operator from
    $H'$ to $H$.  Second, if  $Q_{\alpha,\varepsilon,T}$ is non-negative,
    then the solvability of
    the equation (\ref{yu-12-11-1})
    is equivalent to the finite cost condition of the infinite-horizon LQ problem corresponding to
    (\ref{yu-12-11-1}) (see \cite[Theorem 2.2]{Flandoli-Lasiecka-Triggiani}). Third,  the finite cost condition of the aforementioned
    LQ problem is equivalent to the stabilizability of the system \eqref{yu-12-2-b-1} (see
    \cite[Proposition 3.9]{Liu-Wang-Xu-Yu}).
     Finally, we can only show that
    when $(\varepsilon,T)$  satisfies (\ref{yu-mmmm-1}), the above $Q_{\alpha,\varepsilon,T}$ is non-negative.
    These facts explain why the pair $(\varepsilon,T)$  needs to satisfy (\ref{yu-mmmm-1}) in
    Theorem \ref{yu-theorem-12-16-1}.

    $(ii)$ From the discussions in the note $(i)$, we see that our method
           is to construct directly an operator which satisfies
      a  Riccati equation (related to an infinite-horizon LQ problem)
      instead of solving  a  Riccati equation
       (the latter needs to prove the existence
      of solutions for the corresponding Riccati equation). 

\end{remark}

\section{Further studies}\label{section411}

The quantity $\omega^*$ defined in (\ref{yu-12-19-1}) is called as the best stabilization decay rate
for the system  (\ref{yu-12-2-b-1}). It is the same as that defined by  \cite[(4)]{Trelat-Wang-Xu}.
When the system (\ref{yu-12-2-b-1}) is stabilizable, we have $\omega^*\in(0,+\infty]$. In particular, when $\omega^*=+\infty$, the system (\ref{yu-12-2-b-1}) is completely stabilizable.
Before stating the  main result of this section, we give the next proposition.

\begin{proposition}\label{prop4.111}
    Suppose that  $(H_1)$-$(H_3)$ hold and the system (\ref{yu-12-2-b-1}) is stabilizable. Let $\omega^*\in(0,+\infty]$ be
given by (\ref{yu-12-19-1}). Then for each $\theta\in(0,\omega^*)$,
there is $\overline{C}_1(\theta)\geq 0$ and $\overline{C}_2(\theta)\geq 1$ so that
   (\ref{yu-11-16-6}) holds for $\alpha=\theta$, $C_1(\alpha)=\overline{C}_1(\theta)$ and $C_2(\alpha)=\overline{C}_2(\theta)$, i.e.,
\begin{eqnarray}\label{yu-12-20-1}
    \|S^*(t)\varphi\|_{H'}^2\leq \overline{C}_1(\theta)\int_0^t\|B^*S^*(s)\varphi\|_{U'}^2ds+\overline{C}_2(\theta) e^{-2\theta t}\|\varphi\|_{H'}^2, \;\;t>0, \;\varphi\in D(A^*).
\end{eqnarray}
\end{proposition}
\begin{proof}
    First of all, it follows by (\ref{yu-12-19-1}) and Definition \ref{yu-definition-11-17-1} that for any $\theta\in(0,\omega^*)$,  there is a $C_0$-semigroup $\mathcal{S}_{\theta}(\cdot)$ on $H$ (with the generator $\mathcal{A}_{\theta}: D(\mathcal{A}_{\theta})\subset H\to H$) and an operator
    $K_{\theta}\in\mathcal{L}(D(\mathcal{A}_{\theta});U)$ so that
\begin{enumerate}
  \item [$(b_1)$] there exists  $\widehat{C}_1(\theta)\geq 1$ such that $\|\mathcal{S}_{\theta}(t)\|_{\mathcal{L}(H)}
  \leq \widehat{C}_1(\theta)e^{-\theta t}$ for any  $t\in\mathbb{R}^+$;
\item[$(b_2)$] for any $x\in D(\mathcal{A}_{\theta})$,  $\mathcal{A}_{\theta} x=(\widetilde{A}+BK_{\theta})x$;
  \item [$(b_3)$] there exists  $\widehat{C}_2(\theta)\geq 0$ such that $\|K_{\theta}\mathcal{S}_{\theta}(\cdot)x\|_{L^2(\mathbb{R}^+;U)}
  \leq \widehat{C}_2(\theta)\|x\|_H$ for any $x\in D(\mathcal{A}_{\theta})$.
\end{enumerate}
    Arbitrarily fix  $x\in D(\mathcal{A}_{\theta})$, $\varphi\in D(A^*)$ and $t>0$. Then it follows by $(b_2)$ and (\ref{1.3,202211}) that
\begin{eqnarray*}\label{yu-12-20-5}
    &\;&\frac{d}{ds}\langle \mathcal{S}_{\theta}(s)x,S^*(t-s)\varphi\rangle_{H,H'}\nonumber\\
    &=&\langle \mathcal{A}_{\theta}\mathcal{S}_{\theta}(s)x, S^*(t-s)\varphi\rangle_{H,H'}-\langle\mathcal{S}_{\theta}(s)x, A^*S^*(t-s)\varphi\rangle_{H,H'}\nonumber\\
    &=&\langle (\widetilde{A}+BK_{\theta})\mathcal{S}_{\theta}(s)x, S^*(t-s)\varphi\rangle_{D(A^*)',D(A^*)}-\langle\mathcal{S}_{\theta}(s)x, A^*S^*(t-s)\varphi\rangle_{H,H'}\nonumber\\
    &=&\langle BK_{\theta}\mathcal{S}_{\theta}(s)x, S^*(t-s)\varphi\rangle_{D(A^*)',D(A^*)}=\langle K_{\theta}\mathcal{S}_{\theta}(s)x, B^*S^*(t-s)\varphi\rangle_{U,U'},\;\; s\in(0,t).
\end{eqnarray*}
    By integrating the above equality with respect to $s$ over $[0,t]$, we get
\begin{equation*}\label{yu-12-20-6}
    \langle \mathcal{S}_{\theta}(t)x,\varphi\rangle_{H,H'}
    -\langle x,S^*(t)\varphi\rangle_{H,H'}=\int_0^t\langle K_{\theta}\mathcal{S}_{\theta}(s)x,B^*S^*(t-s)\varphi\rangle_{U,U'}ds.
\end{equation*}
    This, together with $(b_1)$ and $(b_3)$, yields
\begin{equation*}\label{yu-12-20-7}
    |\langle x,S^*(t)\varphi\rangle_{H,H'}|
    \leq \widehat{C}_2(\theta)\|x\|_H\left(\int_0^t\|B^*S^*(t-s)\varphi\|_{U'}^2ds\right)^{\frac{1}{2}}
    +\widehat{C}_1(\theta)e^{-\theta t}\|x\|_H\|\varphi\|_{H'}.
\end{equation*}
Since $t>0$, $\varphi\in D(A^*)$ and $x\in D(\mathcal{A}_{\theta})$ were arbitrarily taken, the above,  along with
    the density of $D(\mathcal{A}_{\theta})$ in $H$, gives
\begin{equation*}\label{yu-12-20-8}
    \|S^*(t)\varphi\|_{H'}^2\leq 2(\widehat{C}_2(\theta))^2\int_0^t\|B^*S^*(s)\varphi\|_{U'}^2ds
    +2(\widehat{C}_1(\theta))^2 e^{-2\theta t}\|\varphi\|_{H'}^2,
    \;\;t>0,\;\varphi\in D(A^*),
\end{equation*}
   which leads to  (\ref{yu-12-20-1}) with $\overline{C}_1(\theta)=2(\widehat{C}_2(\theta))^2$ and $\overline{C}_2(\theta)=2(\widehat{C}_1(\theta))^2$. Thus, we complete the proof of Proposition \ref{prop4.111}.
\end{proof}

\begin{theorem}\label{yu-theorem-12-20-1}
    Assume that $(H_1)$-$(H_3)$ are true and
the system (\ref{yu-12-2-b-1}) is stabilizable. Let $\omega^*\in(0,+\infty]$ be
given by (\ref{yu-12-19-1}).
Let $\overline{C}_1(\theta)\geq 0$ and $\overline{C}_2(\theta)\geq 1$, with $\theta\in(0,\omega^*)$,
be given by Proposition \ref{prop4.111}. Then for each $\mu\in(0,\omega^*)$,  the following  conclusions are true:
\begin{enumerate}
  \item [$(i)$] If $\omega^*\in(0,+\infty)$, then for each
  $T$ satisfying
 \begin{equation}\label{1.9,12-27w}
   (\omega^*-\mu)^{-1}\ln \left[ \overline{C}_2\left(\overline{\theta}\right)\right] < T<+\infty,
   \;\;\mbox{with}\;\;\overline{\theta}:=\frac{1}{2}(\omega^*+\mu),
\end{equation}
      the following operator (from $\Pi_{2\overline{\theta},\overline{\varepsilon},T}[D(A^*)]$ to $U$)
       is a feedback law stabilizing the system (\ref{yu-12-2-b-1})
        with the decay rate
          $\mu$:
\begin{equation}\label{yu-11-20-1000}
    K_{\mu,T}:=-T\overline{C}_1\left(\overline{\theta}\right)e^{2\overline{\theta} T}J_1B^*
    \Pi_{2\overline{\theta},\overline{\varepsilon},T}^{-1},
\end{equation}
      where $\overline{\theta}$ is given in \eqref{1.9,12-27w}, $\overline{\varepsilon}:=T^{-1}\ln \left[ \overline{C}_2\left(\overline{\theta}\right)\right]$
     and
     $\Pi_{2\overline{\theta},\overline{\varepsilon},T}$ is defined by (\ref{yu-11-12-3}) with
     \begin{equation*}
     \alpha=2\overline{\theta};\; \varepsilon=\overline{\varepsilon};\;C_1(\alpha)=\overline{C}_1
     \left(\overline{\theta}\right);\;
     C_2(\alpha)=\overline{C}_2\left(\overline{\theta}\right).
     \end{equation*}

   \item [$(ii)$] If $\omega^*=+\infty$, then for each  $T$ satisfying
\begin{equation}\label{yu-12-30-1}
    \mu^{-1}\ln \left[ \overline{C}_2\left(\theta^*\right)\right]< T<+\infty,\;\;\mbox{with}\;\;\theta^*:=\frac{3\mu}{2},
\end{equation}
    the following operator (from $\Pi_{2\theta^*,\varepsilon^*,T}[D(A^*)]$ to $U$)
    is a feedback law stabilizing the system (\ref{yu-12-2-b-1})
        with the decay rate
          $\mu$:
    \begin{equation}\label{yu-12-26-b-1}
K_{\mu,T}':=-T\overline{C}_1\left(\theta^*\right)e^{2\theta^* T}J_1B^*\Pi_{2\theta^*,\varepsilon^*,T}^{-1},
    \end{equation}
    where $\theta^*$ is given in \eqref{yu-12-30-1}, $\varepsilon^*:=T^{-1}\ln \left[\overline{C}_2\left(\theta^*\right)\right]$
    and
     $\Pi_{2\theta^*,\varepsilon^*,T}$ is defined by (\ref{yu-11-12-3}) with
      \begin{equation*}
      \alpha=2\theta^*;\;  \varepsilon=\varepsilon^*;\;C_1(\alpha)=\overline{C}_1\left(\theta^*\right);
      C_2(\alpha)=\overline{C}_2\left(\theta^*\right).
           \end{equation*}

\end{enumerate}

\end{theorem}
\begin{proof}
Arbitrarily fix $\mu\in(0,\omega^*)$.
    To show the conclusion  $(i)$, we arbitrarily fix $T$ satisfying (\ref{1.9,12-27w})
     and write $\overline{\theta}:=\frac{1}{2}(\omega^*+\mu)$.
    Two observations are given in order. First, it follows from Proposition \ref{prop4.111}
    that  $(H_4)$ holds
    for $\alpha=2\overline{\theta}$, $C_1(\alpha)=\overline{C}_1\left(\overline{\theta}\right)$ and $C_2(\alpha)=\overline{C}_2\left(\overline{\theta}\right)$. Second, by (\ref{1.9,12-27w}),
    one can easily check that
    \begin{equation*}
    \frac{1}{2}\left(2\overline{\theta}-T^{-1}\ln \left[\overline{C}_2\left(\overline{\theta}\right)\right]\right)\geq \mu
    \;\;\mbox{and}\;\;(2\overline{\theta})^{-1}\ln \left[\overline{C}_2\left(\overline{\theta}\right)\right]<T<+\infty.
    \end{equation*}
    From the above two observations  and
     Theorem \ref{yu-theorem-12-3-1}, we see that the operator $K_{\mu,T}: \Pi_{2\overline{\theta},\overline{\varepsilon},T}[D(A^*)]\to U$ defined by (\ref{yu-11-20-1000})
    is a feedback law stabilizing  the system (\ref{yu-12-2-b-1}) with the decay rate $\mu$.
    This completes the proof of $(i)$.
\par
   To show $(ii)$, we arbitrarily fix $T$ satisfying (\ref{yu-12-30-1})
    and write $\theta^*:=\frac{3\mu}{2}$. Two facts are given in order:
   First, it follows from  Proposition \ref{prop4.111}  that  $(H_4)$ holds for $\alpha=2\theta^*$, $C_1(\alpha)=\overline{C}_1\left(\theta^*\right)$ and $C_2(\alpha)=\overline{C}_2\left(\theta^*\right)$.
   Second, by (\ref{yu-12-30-1}), one can directly verify that
    \begin{equation*}
    \frac{1}{2}\left(2\theta^*-T^{-1}\ln \left[\overline{C}_2\left(\theta^*\right)\right]\right)\geq \mu\;\;
    \mbox{and}\;\;(2\theta^*)^{-1}\ln \left[\overline{C}_2\left(\theta^*\right)\right]<T<+\infty.
     \end{equation*}
    From these two facts and  Theorem \ref{yu-theorem-12-3-1},
    we see that
    the operator $K_{\mu, T}':\Pi_{2\theta^*,\varepsilon^*,T}[D(A^*)]\to U$ defined by (\ref{yu-12-26-b-1}) is a feedback law stabilizing the system (\ref{yu-12-2-b-1}) with the decay rate  $\mu$, i.e., $(ii)$ holds.

\par
   Hence, we complete the proof of Theorem \ref{yu-theorem-12-20-1}.
\end{proof}

%\section{Some examples}
%\subsection{Controlled ODEs}
%    We consider the following control system:
%\begin{equation}\label{yu-5-16-1}
%    x'(t)=Ax(t)+Bu(t),\;\;t>0,
%\end{equation}
%    where $A\in\mathbb{R}^{n\times n}$, $B\in\mathbb{R}^{n\times m}$
%    ($n,m\in\mathbb{N}^+$) and $u\in L^2(\mathbb{R}^+;U)$. It is clear that the pair
%    $(A,B)$ satisfies the assumptions $(H_1)$-$(H_3)$. It is well known that the system (\ref{yu-5-16-1}) is stabilizable if and only if the pair $(A,B)$ satisfies
%\begin{enumerate}
%  \item [($H'_4$)] $\mbox{Rank}(\beta I-A,B)=n$ for any $\beta\in\mathbb{C}$ with $\mbox{Re}(\beta)\geq 0$.
%\end{enumerate}
%    (See, for example, \cite{??}.) By \cite[Theorem 1]{Trelat-Wang-Xu} and $(iii)$ in Remark
%    \ref{yu-remark-12-2-2}, the condition $(H_4')$ is equivalent to that there exists $\alpha>0$, $C_1(\alpha)\geq 0$ and $C_2(\alpha)\geq 1$ such that
%\begin{equation}\label{yu-5-16-2}
%    \|e^{A^\top t}\varphi\|^2_{\mathbb{R}^n}\leq C_1(\alpha)\int_0^t\|B^\top e^{A^\top s}\varphi\|_{\mathbb{R}^m}^2ds
%    +C_2(\alpha)e^{-\alpha t}\|\varphi\|^2_{\mathbb{R}^n}\;\;\mbox{for any}\;\;t>0,\;\;
%    \varphi\in \mathbb{R}^n.
%\end{equation}
%    The inequality (\ref{yu-11-16-6}) implies that there is $k\in\mathbb{N}^+$ such that
%\begin{equation}\label{yu-5-16-3}
%    \|e^{A^\top t}\varphi\|^2_{\mathbb{R}^n}\leq k\int_0^t\|B^\top e^{A^\top s}\varphi\|_{\mathbb{R}^m}^2ds
%    +ke^{-k^{-1} t}\|\varphi\|^2_{\mathbb{R}^n}\;\;\mbox{for any}\;\;t>0,\;\;
%    \varphi\in \mathbb{R}^n.
%\end{equation}
%
%
%\subsection{Coupled hyperbolic control systems}
%

\section{Appendices}\label{sec-12-5}
\subsection{Appendix A}\label{yu-subsec-12-9-1}
    In this subsection, we present a direct proof for the equivalence of
    the inequalities (\ref{yu-11-16-5}) and (\ref{yu-11-16-6}).
\begin{proposition}\label{yu-proposition-12-9-1}
    The inequalities (\ref{yu-11-16-5}) and (\ref{yu-11-16-6}) are equivalent.
\end{proposition}
\begin{proof}
We divide the proof by two steps.

\vskip 5pt
\noindent {\it Step 1.
    We first prove $(\ref{yu-11-16-6})\Rightarrow (\ref{yu-11-16-5})$.}

       Let $\alpha>0$, $C_1(\alpha)\geq 0$ and $C_2(\alpha)\geq 1$ be given in (\ref{yu-11-16-6}). Then there is  $\widehat{T}>0$ such that
    $\widehat{\delta}:=C_2(\alpha)e^{-\alpha\widehat{T}}<1$.
    Thus, by taking $t=\widehat{T}$ in  (\ref{yu-11-16-6}), we get (\ref{yu-11-16-5})
    with $T=\widehat{T}$, $C(\alpha,T)=C_1(\alpha)$ and $\delta=\widehat{\delta}$.
\par
   \vskip 5pt
\noindent {\it Step 2. We prove $(\ref{yu-11-16-5})\Rightarrow (\ref{yu-11-16-6})$. }

   Let $T>0$, $\delta\in(0,1)$ and $C(\delta,T)\geq 0$ be given in (\ref{yu-11-16-5}). We
     first claim that
     for any $n\in\mathbb{N}^+$,
    \begin{equation}\label{yu-12-9-2}
    \|S^*(nT)\varphi\|_{H'}^2\leq C(\delta,T)\sum_{j=0}^{n-1}\delta^j
    \int_0^{nT}\|B^*S^*(s)\varphi\|_{U'}^2ds +\delta^n\|\varphi\|_{H'}^2.
\end{equation}
    Indeed, (\ref{yu-11-16-5}) gives \eqref{yu-12-9-2} with $n=1$.
    Suppose that \eqref{yu-12-9-2}, with $n=k$, is true. Then, by (\ref{yu-11-16-5}) and  the time-invariance of the system (\ref{yu-12-2-b-1}), we have
    \begin{eqnarray*}\label{yu-12-9-3}
    \|S^*((k+1)T)\varphi\|_{H'}^2&\leq& C(\delta,T)\sum_{j=0}^{k-1}\delta^j\int_T^{(k+1)T}
    \|B^*S^*(s)\varphi\|_{U'}^2ds+\delta^k\|S^*(T)\varphi\|_{H'}^2\nonumber\\
    &\leq&C(\delta,T)\sum_{j=0}^{k}\delta^j\int_0^{(k+1)T}\|B^*S^*(s)\varphi\|_{U'}^2ds
    +\delta^{k+1}\|\varphi\|_{H'}^2,
\end{eqnarray*}
    which leads to \eqref{yu-12-9-2} with $n=k+1$. So by the induction, \eqref{yu-12-9-2} holds for all
    $n\in\mathbb{N}^+$.

    Next, we let $\alpha=-T^{-1}\ln \delta$ (which implies $\alpha>0$ and $\delta=e^{-\alpha T}$). Then by
    \eqref{yu-12-9-2}, we have that for any $n\in\mathbb{N}^+$,
    \begin{eqnarray}\label{yu-12-9-4}
    \|S^*(nT)\varphi\|^2_{H'}&\leq& (1-\delta)^{-1}C(\delta,T)\int_0^{nT}
    \|B^*S^*(s)\varphi\|_{U'}^2ds+\delta^n\|\varphi\|_{H'}^2\nonumber\\
    &=&(1-e^{-\alpha T})^{-1}C(e^{-\alpha T},T)\int_0^{nT}\|B^*S^*(s)\varphi\|_{U'}^2ds
    +e^{-n\alpha T}\|\varphi\|_{H'}^2.
\end{eqnarray}

    We now arbitrarily fix $t\in\mathbb{R}^+$. Then  there is
    $m\in\mathbb{N}$ such that
    \begin{equation}\label{5.3w}
    mT\leq t<(m+1)T.
    \end{equation}
      In the case that $m=0$ (i.e., $t\in[0,T)$), we have
\begin{equation}\label{yu-12-9-5}
    \|S^*(t)\varphi\|_{H'}^2\leq \widehat{C}_2(\alpha,T)e^{-\alpha t}
    \|\varphi\|_{H'}^2,\;\;\mbox{with}\;\;\widehat{C}_2(\alpha,T):=\big(\sup_{\sigma\in[0,T]}
    \|S^*(\sigma)\|_{\mathcal{L}(H')}\big)^2e^{\alpha T}.
\end{equation}
    In the case that  $m\in\mathbb{N}^+$, it follows by (\ref{yu-12-9-4}) and \eqref{5.3w}
    that
\begin{eqnarray}\label{5.5w}
    \|S^*(t)\varphi\|_{H'}^2&=&\|S^*(t-mT)S^*(mT)\varphi\|_{H'}^2\leq
    \big(\sup_{\sigma\in[0,T]}\|S^*(\sigma)\|_{\mathcal{L}(H')}\big)^2
    \|S^*(mT)\varphi\|_{H'}^2\nonumber\\
    &\leq&\widehat{C}_1(\alpha,T)\int_0^t\|B^*S^*(s)\varphi\|_{U'}^2ds
    +\widehat{C}_2(\alpha,T)e^{-\alpha t}\|\varphi\|_{H'}^2,
\end{eqnarray}
where $\widehat{C}_2(\alpha,T)$ is given in (\ref{yu-12-9-5}) and
    \begin{equation*}
    \widehat{C}_1(\alpha,T):=\big(\sup_{\sigma\in[0,T]}\|S^*(\sigma)\|_{\mathcal{L}(H')}
    \big)^2(1-e^{-\alpha T})^{-1}C_1(e^{-\alpha T},T).
    \end{equation*}

    Finally, \eqref{yu-12-9-5} and \eqref{5.5w} leads to (\ref{yu-11-16-6})  with
    $C_1(\alpha)=\widehat{C}_1(\alpha,T)$ and $C_2(\alpha)=\widehat{C}_2(\alpha,T)$.

    Hence, we finish the proof of Proposition \ref{yu-proposition-12-9-1}.
\end{proof}

\subsection{Appendix B}\label{yu-appen-1}
\begin{proof}[The proof of Lemma \ref{yu-lemma-11-23-2}]

Arbitrarily fix  $\gamma>0$, $M\in \mathcal{L}(H')$  and $\varphi\in D(A^*)$.
The proof is divided into two steps.
\vskip 5pt
\noindent\emph{Step 1. We show that for each  $t\in [-\gamma,\gamma]$,  $w(t;\varphi)\in D(A^*)$ and
\begin{eqnarray}\label{yu-11-25-1}
    A^*w(t;\varphi)=M\mathcal{V}_{\alpha,\varepsilon,T}(t)\varphi-S^*(-t)M\varphi
    -\int_0^tS^*(s-t)M\mathcal{V}_{\alpha,\varepsilon,T}(s)\Delta_{\alpha,\varepsilon,T}\varphi ds.
\end{eqnarray}
(Recall that $\mathcal{V}_{\alpha,\varepsilon,T}(\cdot)$ and $\Delta_{\alpha,\varepsilon,T}$
are given in Remark \ref{reamrk2.3w}.)
}

   To this end, we arbitrarily fix  $t\in[-\gamma,\gamma]$. By
    the definition of $w(\cdot;\varphi)$, we have that for each $h\in(0,h_0)$ ($h_0>0$ is fixed arbitrarily),
\begin{eqnarray}\label{yu-11-25-2}
    \frac{S(h)-I}{h}w(t;\varphi)&=&\frac{1}{h}
    \left(\int_0^tS^*(s+h-t)M\mathcal{V}_{\alpha,\varepsilon,T}(s)\varphi ds
    -\int_0^tS^*(s-t)M\mathcal{V}_{\alpha,\varepsilon,T}(s)\varphi ds\right)\nonumber\\
    &=& \mathcal{I}_1(h)+\mathcal{I}_2(h)+\mathcal{I}_3(h),
\end{eqnarray}
    where
\begin{equation*}
\begin{cases}
    \mathcal{I}_1(h):=\frac{1}{h}\int_t^{t+h}S^*(s-t)M
    \mathcal{V}_{\alpha,\varepsilon,T}(s-h)\varphi ds,\\
    \mathcal{I}_2(h):=-\frac{1}{h}\int_0^hS^*(s-t)M\mathcal{V}_{\alpha,\varepsilon,T}(s-h)\varphi ds,\\
    \mathcal{I}_3(h):=\frac{1}{h}\int_0^tS^*(s-t)M\mathcal{V}_{\alpha,\varepsilon,T}(s)(
    \mathcal{V}_{\alpha,\varepsilon,T}(-h)-I)\varphi ds.
\end{cases}
\end{equation*}

With respect to $ \mathcal{I}_1(h)$, we claim
   \begin{equation}\label{yu-11-25-3}
    \lim_{h\to 0^+}\mathcal{I}_1(h)=M\mathcal{V}_{\alpha,\varepsilon,T}(t)\varphi.
\end{equation}
    Indeed, we have the following facts: First, it is obvious that
\begin{eqnarray}\label{yu-11-25-4}
    \mathcal{I}_1(h)=\frac{1}{h}\int_t^{t+h}S^*(s-t)M(\mathcal{V}_{\alpha,\varepsilon,T}(s-h)-
    \mathcal{V}_{\alpha,\varepsilon,T}(t))\varphi ds
    +\frac{1}{h}\int_t^{t+h}S^*(s-t)M\mathcal{V}_{\alpha,\varepsilon,T}(t)\varphi ds.
\end{eqnarray}
    Second, it follows  from the strong continuity of $S^*(\cdot)$ that
\begin{equation}\label{yu-11-25-5}
    \lim_{h\to 0^+}\frac{1}{h}\int_t^{t+h}S^*(s-t)M\mathcal{V}_{\alpha,\varepsilon,T}(t)\varphi ds
    =M\mathcal{V}_{\alpha,\varepsilon,T}(t)\varphi.
\end{equation}
    Third, direct computations show
\begin{eqnarray}\label{yu-11-25-6}
    &\;&\left\|\frac{1}{h}\int_t^{t+h}S^*(s-t)M(\mathcal{V}_{\alpha,\varepsilon,T}(s-h)-
    \mathcal{V}_{\alpha,\varepsilon,T}(t))\varphi ds\right\|_{H'}\nonumber\\
    &=&\left\|\frac{1}{h}\int_t^{t+h} S^*(s-t)M\int_{s-h}^t\mathcal{V}_{\alpha,\varepsilon,T}(\sigma)\Delta_{\alpha,\varepsilon,T}
    \varphi d\sigma ds\right\|_{H'}\nonumber\\
    &\leq& h\sup_{\sigma\in[0,h_0]}\|S^*(\sigma)\|_{\mathcal{L}(H')}
    \|M\|_{\mathcal{L}(H')} \sup_{\sigma\in[t-h_0,t]}\|\mathcal{V}_{\alpha,\varepsilon,T}
    (\sigma)\|_{\mathcal{L}(H')}\|\Delta_{\alpha,\varepsilon,T}\varphi\|_{H'}\to0\;\;\mbox{as}\;\;h\to 0^+.
\end{eqnarray}
    (Here, we used  that $\sup_{s\in[t,t+h]}|t-s+h|=h$.)
    Now, (\ref{yu-11-25-3}) follows by (\ref{yu-11-25-4}), (\ref{yu-11-25-5}) and (\ref{yu-11-25-6}) at once.

    With respect to $ \mathcal{I}_2(h)$, we can use a very similar way to that used in the proof of  (\ref{yu-11-25-3}) to find
\begin{equation}\label{yu-11-25-7}
    \lim_{h\to 0^+}\mathcal{I}_2(h)=-S^*(-t)M\varphi.
\end{equation}

 With respect to $ \mathcal{I}_3(h)$, we claim
\begin{equation}\label{yu-11-25-8}
    \lim_{h\to0^+}\mathcal{I}_3(h)=-\int_0^t S^*(s-t)M\mathcal{V}_{\alpha,\varepsilon,T}(s)
    \Delta_{\alpha,\varepsilon,T}\varphi ds.
\end{equation}
    For this purpose, several facts are given in order: First, since $\Delta_{\alpha,\varepsilon,T}$
    is the generator of the $C_0$-group $\mathcal{V}_{\alpha,\varepsilon,T}(\cdot)$ and
    $\varphi\in D(A^*)(=D(\Delta_{\alpha,\varepsilon,T}))$, we have
\begin{equation}\label{yu-11-25-9}
    \mathcal{I}_3(h)=-\int_0^tS^*(s-t)M\mathcal{V}_{\alpha,\varepsilon,T}(s)\left(\frac{1}{h}
    \int_{-h}^0\mathcal{V}_{\alpha,\varepsilon,T}(\sigma)\Delta_{\alpha,\varepsilon,T}\varphi
    d\sigma\right)ds.
\end{equation}
    Second, direct computations show that  for each $s\in[-|t|,|t|]$,
\begin{eqnarray}\label{yu-11-25-10w}
    &\;&\left\|S^*(s-t)M\mathcal{V}_{\alpha,\varepsilon,T}(s)
    \left(\frac{1}{h}\int_{-h}^0\mathcal{V}_{\alpha,\varepsilon,T}(\sigma)
    \Delta_{\alpha,\varepsilon,T}\varphi d\sigma\right)\right\|_{H'}\nonumber\\
    &\leq&\sup_{\sigma\in[-2|t|,2|t|]}\|S^*(\sigma)\|_{\mathcal{L}(H')}
    \|M\|_{\mathcal{L}(H')}\sup_{\sigma\in[-2|t|-h_0,2|t|]}\|\mathcal{V}_{\alpha,\varepsilon,T}
    (\sigma)\|_{\mathcal{L}(H')}\|\Delta_{\alpha,\varepsilon,T}\varphi\|_{H'}.
\end{eqnarray}
    Third, the strong continuity of $\mathcal{V}_{\alpha,\varepsilon,T}(\cdot)$ leads to
\begin{equation}\label{yu-11-25-11w}
    \lim_{h\to0^+}\frac{1}{h}
    \int_{-h}^0\mathcal{V}_{\alpha,\varepsilon,T}(\sigma)\Delta_{\alpha,\varepsilon,T}\varphi
    d\sigma=\Delta_{\alpha,\varepsilon,T}\varphi.
\end{equation}
   Now, by \eqref{yu-11-25-10w}
   and \eqref{yu-11-25-11w}, we can apply  the dominated convergence theorem in
   \eqref{yu-11-25-9} to get  (\ref{yu-11-25-8}).
\par
   Finally, it follows from  (\ref{yu-11-25-2}), (\ref{yu-11-25-3}), (\ref{yu-11-25-7}) and (\ref{yu-11-25-8}) that
\begin{equation*}\label{yu-11-25-12}
    \lim_{h\to 0^+}\frac{S^*(h)-I}{h}w(t;\varphi)\;\;\mbox{exists, i.e., }\;\;w(t;\varphi)\in D(A^*)
\end{equation*}
    (see \cite[Chapter 1, Section 1.1]{Pazy}) and that (\ref{yu-11-25-1}) holds.
\vskip 5pt
   \noindent \emph{Step 2. We prove  (\ref{yu-11-25-bbb-1}).}

   By the note $(ii)$ in Remark \ref{yu-remark-12-2-1} and Step 1, we  have
\begin{eqnarray}\label{yu-11-25-13}
    &\;&\int_{-\gamma}^\gamma\|B^*w(t;\varphi)\|_{U'}^2dt=
    \int_{-\gamma}^\gamma\|E^*(\bar{\lambda}I+A^*)w(t;\varphi)\|^2_{U'}dt\nonumber\\
    &\leq&16\|E^*\|_{\mathcal{L}(H';U')}^2
    \int_{-\gamma}^\gamma\biggl(|\lambda|^2\|w(t;\varphi)\|_{H'}^2
    +\|M\mathcal{V}_{\alpha,\varepsilon,T}(t)\varphi\|^2_{H'}
    +\|S^*(-t)M\varphi\|^2_{H'}\nonumber\\
    &\;&+\left\|\int_0^tS^*(s-t)M\mathcal{V}_{\alpha,\varepsilon,T}(s)
    \Delta_{\alpha,\varepsilon,T}\varphi ds\right\|^2_{H'}\biggl)dt.
\end{eqnarray}
   (Here, $\bar{\lambda}$ is the conjugate of $\lambda$.)
   Since
\begin{equation*}\label{yu-11-25-14}
\begin{cases}
    \sup_{t\in[-\gamma,\gamma]}\|w(t;\varphi)\|_{H'}\leq
    \sup_{\sigma\in[-\gamma,\gamma]}\|S^*(\sigma)\|_{\mathcal{L}(H')}
    \int_{-\gamma}^{\gamma}\|M\mathcal{V}_{\alpha,\varepsilon,
    T}(s)\varphi\|_{H'}ds,\\
    \sup_{t\in[-\gamma,\gamma]}\|S^*(-t)M\varphi\|_{H'}
    \leq \sup_{t\in[-\gamma,\gamma]}\|S^*(t)\|_{\mathcal{L}(H')}\|M\varphi\|_{H'},\\
    \sup_{t\in[-\gamma,\gamma]}\left\|\int_0^tS^*(s-t)M\mathcal{V}_{\alpha,\varepsilon,T}(s)
    \Delta_{\alpha,\varepsilon,T}\varphi ds\right\|_{H'}\\
    \;\;\;\;\;\;\;\;\;\;\;\leq \sup_{\sigma\in[-\gamma,\gamma]}\|S^*(\sigma)\|_{\mathcal{L}(H')}
    \int_{-\gamma}^{\gamma}\|M\mathcal{V}_{\alpha,\varepsilon,
    T}(s)\Delta_{\alpha,\varepsilon,T}\varphi\|_{H'}ds,
\end{cases}
\end{equation*}
   applying  the H\"{o}lder inequality in (\ref{yu-11-25-13}) leads to
    (\ref{yu-11-25-bbb-1}).
\par
\vskip 5pt

By Steps 1 and Step 2, we finish the proof of Lemma \ref{yu-lemma-11-23-2}.
\end{proof}
\subsection{Appendix C}\label{yu-appen-2}
\begin{proof}[The proof of Lemma \ref{yu-lemma-11-23-1}]
Arbitrarily fix $t\in\mathbb{R}$ and $\varphi,\psi\in D(A^*)$.
 The proof is divided into two steps.
\vskip 5pt
  \noindent  \emph{Step 1. We prove
    \begin{eqnarray}\label{yu-11-23-5}
    &\;&\langle\Pi_{\alpha,\varepsilon,T}\varphi,\psi\rangle_{H,H'}-\mathcal{E}(t)\nonumber\\
    &=&\langle \Pi_{\alpha,\varepsilon,T}\mathcal{V}_{\alpha,\varepsilon,T}(t)\varphi,S^*(-t)\psi\rangle_{H,H'}
    +TC_1(\alpha)e^{\alpha T}\int_0^t\langle J_1B^*\mathcal{V}_{\alpha,\varepsilon,T}(s)\varphi,
    B^*S^*(-s)\psi\rangle_{U,U'}ds,
\end{eqnarray}
    where
    $\mathcal{E}(t)=\mathcal{E}_1(t)+\mathcal{E}_2(t)+\mathcal{E}_3(t)+\mathcal{E}_4(t)$
    with
\begin{equation*}\label{yu-11-23-7}
\begin{cases}
    \mathcal{E}_1(t):=\int_0^t\langle\Pi_{\alpha,\varepsilon,T} S^*(s-t)\Pi_{\alpha,\varepsilon,T}^{-1}
    \widehat{Q_{\alpha,\varepsilon,T}}\mathcal{V}_{\alpha,\varepsilon,T}(s)\varphi,
    S^*(-t)\psi\rangle_{H,H'}ds,\\
    \mathcal{E}_2(t):=TC_1(\alpha)e^{\alpha T}\int_0^t\left\langle J_1B^*\left(\int_0^s S^*(\gamma-s)\Pi_{\alpha,\varepsilon,T}^{-1}\widehat{Q_{\alpha,\varepsilon,T}}
    \mathcal{V}_{\alpha,\varepsilon,T}(\gamma)\varphi d\gamma\right),B^*S^*(-s)\psi\right\rangle_{U,U'}ds,\\
    \mathcal{E}_3(t):=-\int_0^t\langle \widehat{Q_{\alpha,\varepsilon,T}}\mathcal{V}_{\alpha,\varepsilon,T}(s)
    \varphi,S^*(-s)\psi\rangle_{H,H'}ds,\\
    \mathcal{E}_4(t):=-\int_0^t\left\langle\widehat{Q_{\alpha,\varepsilon,T}}
    \left(\int_0^sS^*(\gamma-s)\Pi_{\alpha,\varepsilon,T}^{-1}
    \widehat{Q_{\alpha,\varepsilon,T}}
    \mathcal{V}_{\alpha,\varepsilon,T}(\gamma)\varphi d\gamma\right),S^*(-s)\psi\right\rangle_{H,H'}ds.
\end{cases}
\end{equation*}
    (Here $\widehat{Q_{\alpha,\varepsilon,T}}$ is defined by (\ref{yu-11-23-3}).)
}

First of all, it is obvious that $\mathcal{E}_1(t)$, $\mathcal{E}_3(t)$ and $\mathcal{E}_4(t)$
are well defined  since
    $\Pi_{\alpha,\varepsilon,T}^{-1}\widehat{Q_{\alpha,\varepsilon,T}}\in\mathcal{L}(H')$ and $\varphi,\psi\in D(A^*)$. Second, by Lemma \ref{yu-lemma-11-23-2} and the assumption $(H_3)$ that $\mathcal{E}_2(t)$ is well defined since $\varphi,\psi\in D(A^*)$.

    Next, it follows  from Proposition \ref{yu-proposition-11-12-1} that for any $s\in\mathbb{R}$,
\begin{eqnarray}\label{yu-11-15-102}
    &\;&\langle \Pi_{\alpha,\varepsilon,T}A^*S^*(-s)\varphi,S^*(-s)\psi\rangle_{H,H'}
    +\langle \Pi_{\alpha,\varepsilon,T}S^*(-s)\varphi, A^*S^*(-s)\psi\rangle_{H,H'}\nonumber\\
    &=&TC_1(\alpha)e^{\alpha T}\langle J_1B^*S^*(-s)\varphi,B^*S^*(-s)\psi\rangle_{U,U'}
    -\langle\widehat{Q_{\alpha,\varepsilon,T}}
    S^*(-s)\varphi,S^*(-s)\psi\rangle_{H,H'}.
\end{eqnarray}
     Then, since
    \begin{eqnarray*}\label{yu-11-15-103}
    \langle \Pi_{\alpha,\varepsilon,T}A^*S^*(-s)\varphi,S^*(-s)\psi\rangle_{H,H'}
    +\langle \Pi_{\alpha,\varepsilon,T}S^*(-s)\varphi, A^*S^*(-s)\psi\rangle_{H,H'}\nonumber\\
    =-\left[\frac{d}{d\gamma}\langle \Pi_{\alpha,\varepsilon,T}S^*(-\gamma)\varphi,S^*(-\gamma)\psi\rangle_{H,H'}
    \right]_{\gamma=s},\;\;s\in\mathbb{R},
\end{eqnarray*}
    we get, by  integrating (\ref{yu-11-15-102}) with respect to $s$ over $(0,t)$, that
\begin{eqnarray}\label{yu-11-15-104}
    &\;&
    \langle \Pi_{\alpha,\varepsilon,T}\varphi,\psi\rangle_{H,H'}\nonumber\\
    &=&\langle \Pi_{\alpha,\varepsilon,T}S^*(-t)\varphi,S^*(-t)\psi\rangle_{H,H'}
    +TC_1(\alpha)e^{\alpha T}\int_0^t\langle J_1B^*S^*(-s)\varphi,B^*S^*(-s)\psi\rangle_{U,U'}ds\nonumber\\
    &\;&-\int_0^t\langle \widehat{Q_{\alpha,\varepsilon,T}}S^*(-s)\varphi,S^*(-s)\psi\rangle_{H,H'}ds.
\end{eqnarray}
    Finally, (\ref{yu-11-23-5}) follows from \eqref{yu-11-15-104}
    and (\ref{yu-11-23-1}).

\vskip 5pt
  \noindent  \emph{Step 2.  We show
\begin{equation}\label{yu-11-24-1-bbbb}
    \mathcal{E}(t)=0.
\end{equation}}
     When this is done, (\ref{yu-11-23-2}) follows from \eqref{yu-11-24-1-bbbb}
     and (\ref{yu-11-23-5}) at once.

     The remainder is to show \eqref{yu-11-24-1-bbbb}.
     Let
     $n^*\in\mathbb{N}^+$ be such that  $nI-A^*$ is invertible for all $n\geq n^*$.
 We define, for each $n\geq n^*$,
\begin{eqnarray*}\label{yu-11-23-8}
    \mathcal{F}_n(t):=\int_0^t\langle \Pi_{\alpha,\varepsilon,T}S^*(s-t)\mathcal{R}_n\Pi_{\alpha,\varepsilon,T}^{-1}
    \widehat{Q_{\alpha,\varepsilon,T}}
    \mathcal{V}_{\alpha,\varepsilon,T}(s)\varphi,S^*(s-t)S^*(-s)\psi\rangle_{H,H'}ds,
\end{eqnarray*}
    where $\{\mathcal{R}_n\}_{n\geq n^*}$ are given by (\ref{yu-kkk-1}).
     By (\ref{yu-11-15-104}), we find
\begin{equation}\label{yu-11-23-8-1}
    \mathcal{F}_n(t)=\mathcal{F}_{n,1}(t)+\mathcal{F}_{n,2}(t)+\mathcal{F}_{n,3}(t),\; n\geq n^*,
\end{equation}
    where
\begin{equation*}\label{yu-11-23-9}
\begin{cases}
    \mathcal{F}_{n,1}(t):=\int_0^t\langle
    \Pi_{\alpha,\varepsilon,T}\mathcal{R}_n\Pi_{\alpha,\varepsilon,T}^{-1}
    \widehat{Q_{\alpha,\varepsilon,T}}\mathcal{V}_{\alpha,\varepsilon,T}(s)\varphi, S^*(-s)\psi\rangle_{H,H'}ds,\\
    \mathcal{F}_{n,2}(t):=-TC_1(\alpha)e^{\alpha T}\int_0^t
    \int_0^{t-s}\langle J_1B^*S^*(-\gamma)\mathcal{R}_n\Pi_{\alpha,\varepsilon,T}^{-1}
    \widehat{Q_{\alpha,\varepsilon,T}}
    \mathcal{V}_{\alpha,\varepsilon,T}
    (s)\varphi,B^*S^*(-(\gamma+s))\psi\rangle_{U,U'}d\gamma ds,\\
    \mathcal{F}_{n,3}(t):=\int_0^t\int_0^{t-s}\langle \widehat{Q_{\alpha,\varepsilon,T}}S^*(-\gamma)\mathcal{R}_n\Pi_{\alpha,\varepsilon,T}^{-1}
    \widehat{Q_{\alpha,\varepsilon,T}}\mathcal{V}_{\alpha,\varepsilon,T}(s)\varphi,
    S^*(-(\gamma+s))\psi\rangle_{H,H'}d\gamma ds.
\end{cases}
\end{equation*}
Several facts are given in order: First,
since $\{\mathcal{R}_n\}_{n\geq n^*}$ is uniformly bounded (see (\ref{yu-11-24-2})),
we can use (\ref{yu-kkkk-2}) and  the dominated convergence theorem to find
\begin{equation}\label{yu-11-23-10}
    \lim_{n\to+\infty}\mathcal{F}_n(t)=\mathcal{E}_1(t)\;\;
    \mbox{and}\;\;\lim_{n\to+\infty}\mathcal{F}_{n,1}(t)=-\mathcal{E}_3(t).
\end{equation}
Second, direct computations show     that when $n\geq n^*$,
\begin{eqnarray*}\label{yu-11-23-11}
    \mathcal{F}_{n,3}(t)&=&\int_0^t\int_s^t\langle \widehat{Q_{\alpha,\varepsilon,T}}S^*(s-\sigma)\mathcal{R}_n\Pi_{\alpha,\varepsilon,T}^{-1}
    \widehat{Q_{\alpha,\varepsilon,T}}\mathcal{V}_{\alpha,\varepsilon,T}(s)\varphi,
    S^*(-\sigma)\psi\rangle_{H,H'}d\sigma ds\nonumber\\
    &=& \int_0^t\int_0^\sigma\langle \widehat{Q_{\alpha,\varepsilon,T}}S^*(s-\sigma)\mathcal{R}_n\Pi_{\alpha,\varepsilon,T}^{-1}
    \widehat{Q_{\alpha,\varepsilon,T}}\mathcal{V}_{\alpha,\varepsilon,T}(s)\varphi,
    S^*(-\sigma)\psi\rangle_{H,H'}dsd\sigma.
\end{eqnarray*}
    This, along with (\ref{yu-kkkk-2}) and the dominated convergence theorem, yields
\begin{equation}\label{yu-11-23-11-b}
    \lim_{n\to+\infty}\mathcal{F}_{n,3}(t)=-\mathcal{E}_4(t).
\end{equation}
\par
   We now claim
\begin{equation}\label{yu-11-24-1}
    \lim_{n\to+\infty}\mathcal{F}_{n,2}(t)=-\mathcal{E}_2(t).
\end{equation}
    To this end, we  define, for each
     $n\geq n^*$,
\begin{equation}\label{yu-11-24-15}
    \mathcal{H}_n(t):=\int_0^t\left\langle J_1B^*\left(\int_0^s S^*(\gamma-s)\mathcal{R}_n\Pi_{\alpha,\varepsilon,T}^{-1}\widehat{Q_{\alpha,\varepsilon,T}}
    \mathcal{V}_{\alpha,\varepsilon,T}(\gamma)\varphi d\gamma\right),B^*S^*(-s)\psi\right\rangle_{U,U'}ds.
\end{equation}
   Then we have two observations: First, since $\mathcal{R}_nS^*(\cdot)=S^*(\cdot)\mathcal{R}_n$ and $B^*\mathcal{R}_n\in \mathcal{L}(H')$ (see \eqref{yu-11-24-16w}), we have
\begin{equation*}\label{yu-11-24-17}
    \mathcal{H}_n(t)=
    \int_0^t\int_0^s\langle J_1B^*S^*(\gamma-s)\mathcal{R}_n\Pi_{\alpha,\varepsilon}^{-1}
    \widehat{Q_{\alpha,\varepsilon,T}}\mathcal{V}_{\alpha,\varepsilon,T}(\gamma)\varphi,
    B^*S^*(-s)\psi\rangle_{U,U'}d\gamma ds,\;n\geq n^*.
\end{equation*}
    (See \cite[Lemma 11.45]{Aliprantis-Border}.) By  using some simple integral transformations in the above, we find
\begin{equation*}\label{yu-11-24-18}
    \mathcal{H}_n(t)=\int_0^t\int_0^{t-\gamma}\langle J_1B^*S^*(-\sigma)\mathcal{R}_n\Pi_{\alpha,\varepsilon,T}^{-1}
    \widehat{Q_{\alpha,\varepsilon,T}}\mathcal{V}_{\alpha,\varepsilon,T}(\gamma)\varphi,B^*S^*
    (-(\gamma+\sigma))\psi\rangle_{U,U'}d\sigma d\gamma,\;n\geq n^*,
\end{equation*}
   which implies
   \begin{equation}\label{yu-11-24-19}
    \mathcal{F}_{n,2}(t)=-TC_1(\alpha)e^{\alpha T}\mathcal{H}_n(t)
    \;\;\mbox{for each}\;\;n\geq n^*.
\end{equation}
   Second,  if we let, for each $n\geq n^*$,
\begin{equation}\label{yu-11-24-20}
    w_n(s):=\int_0^sS^*(\gamma-s)(\mathcal{R}_n-I)\Pi_{\alpha,\varepsilon,T}^{-1}
    \widehat{Q_{\alpha,\varepsilon,T}}\mathcal{V}_{\alpha,\varepsilon,T}(\gamma)\varphi d\gamma,
\end{equation}
    then, similar to the proof of (\ref{12-18-27w}), we can show
\begin{equation}\label{yu-b0-11-26-1}
    \lim_{n\to+\infty}\int_{-|t|}^{|t|}\|B^*w_n(s)\|_{U'}^2ds
    =0.
\end{equation}
        Now, it follows  by the assumption $(H_3)$, (\ref{yu-11-24-15}), (\ref{yu-11-24-20}) and (\ref{yu-b0-11-26-1}) that
\begin{equation}\label{yu-11-24-22}
    \lim_{n\to+\infty}\mathcal{H}_n(t)
    =\int_0^t\left\langle J_1B^*\left(\int_0^s S^*(\gamma-s)\Pi_{\alpha,\varepsilon,T}^{-1}\widehat{Q_{\alpha,\varepsilon,T}}
    \mathcal{V}_{\alpha,\varepsilon,T}(\gamma)\varphi d\gamma\right),B^*S^*(-s)\psi\right\rangle_{U,U'}ds,
\end{equation}
    which, along with the definition of $\mathcal{E}_2(t)$, (\ref{yu-11-24-19}) and (\ref{yu-11-24-22}),
    yields  (\ref{yu-11-24-1}).
    Finally, \eqref{yu-11-24-1-bbbb} follows by  (\ref{yu-11-23-8-1}), (\ref{yu-11-23-10}), (\ref{yu-11-23-11-b}) and
    (\ref{yu-11-24-1}).

Thus, we  finish the proof of Lemma \ref{yu-lemma-11-23-1}.
\end{proof}

\end{document}